\documentclass{amsart}

\usepackage{amssymb}
\usepackage{graphicx}
\usepackage{epsfig}

\newcommand{\dz}{\partial_z}
\newcommand{\eps}{\varepsilon}
\newcommand{\bU}{\overline{U}}
\newcommand{\cC}{{\mathcal C}}
\newcommand{\cL}{{\mathcal L}}
\newcommand{\cT}{{\mathcal T}}
\newcommand{\cc}{\mbox{\textnormal c.c.}}
\newcommand{\cst}{\mbox{\textnormal Cst }}
\newcommand{\dt}{\partial_t}
\newcommand{\dsp}{\displaystyle}
\newcommand{\C}{{\mathbb C}}
\newcommand{\R}{{\mathbb R}}
\newcommand{\N}{{\mathbb N}}
\newcommand{\T}{{\mathbb T}}
\newcommand{\Z}{{\mathbb Z}}
\newcommand{\init}{_{\vert_{t=0}}}
\newcommand{\uk}{\underline{{\bf k}}}
\newcommand{\uom}{\underline{\omega}}
\newcommand{\bk}{{\mathbf k}}
\newcommand{\curl}{\mbox{\textnormal{curl} }}

\newtheorem{assu}{Assumption}
\newtheorem{prop}{Proposition}
\newtheorem{coro}{Corollary}
\newtheorem{lemm}{Lemma}
\newtheorem{theo}{Theorem}
\newtheorem{exam}{Example}
\newtheorem{rema}{Remark}

\begin{document}
\title{Short Pulses Approximations in Dispersive Media}
\author[M. Colin]{Mathieu Colin}
\address{Universit\'e Bordeaux I; IMB et CNRS UMR 5251\\
351 Cours de la Lib\'eration,\\33405 Talence Cedex, France}
\email{mcolin@math.u-bordeaux1.fr}
\author[D. Lannes]{David Lannes}
\address{Universit\'e Bordeaux I; IMB et CNRS UMR 5251\\
351 Cours de la Lib\'eration,\\33405 Talence Cedex, France}
\email{lannes@math.u-bordeaux1.fr}
\thanks{This work was supported by the ACI Jeunes Chercheurs et Jeunes Chercheuses ``Dispersion et nonlin\'earit\'es''}
\begin{abstract}
We derive various approximations for the solutions of nonlinear hyperbolic
systems with fastly oscillating initial data. We first provide error
estimates for the so-called slowly varying envelope, full dispersion, and 
Schr\"odinger
approximations in a Wiener algebra; this functional framework 
allows us to give precise
conditions on the validity of these models; we give in particular a 
rigorous proof of the ``practical rule'' which serves as a criterion for
the use of the slowly varying envelope approximation (SVEA). We also
 discuss the extension of these models
to short pulses and more generally to large spectrum waves, such as chirped
pulses. We then derive and justify rigorously 
a modified Schr\"odinger equation with improved
frequency dispersion. Numerical computations are then presented, which
confirm the theoretical predictions.

\end{abstract}

\maketitle
\section{Introduction}

\subsection{General setting}

The nonlinear Schr\"odinger equation has been derived as an asymptotic 
model for many physical problems involving the propagation of slowly
modulated oscillating plane waves. Typical examples are water-waves 
\cite{HO,Zakharov}
or ferromagnetism \cite{LM} and of course nonlinear optics 
(see e.g. \cite{Boyd,NM})
for which it plays a central role. It is aimed at approximating the 
solution $u$ of a nonlinear hyperbolic system with fast oscillating initial
condition, say, 
\begin{equation}\label{eqd}
	\left\lbrace
	\begin{array}{l}
	\dsp\dt u+ A(\partial)u+\frac{1}{\eps}Eu=\eps F(u),\vspace{1mm}\\
	\dsp u_{\vert_{t=0}}=U^0(x)e^{i\frac{{\bf k}\cdot x}{\eps}}+\cc;
	\end{array}\right.
\end{equation}
here $\eps\ll 1$ is a small parameter corresponding to the wavelength of
the oscillations in dimensionless variables (see Assumption
\ref{as1} below for more precisions). The so-called
Schr\"odinger approximation can be decomposed into two steps:
\begin{enumerate}
	\item[-a-] 
	\emph{Slowly Varying Envelope Approximation (SVEA)}: one writes 
	the solution $u$ as the product of a fast oscillating wave train
	and an envelope $U$:
	\begin{equation}\label{appr}
		u(t,x)\sim U(t,x)
	e^{i\frac{{\bf k}\cdot x-\omega({\bf k}) t}{\eps}}+\cc,
	\end{equation}
	where $(\omega({\bf k}),{\bf k})$ solves the \emph{dispersion relation}
	(or, equivalently, belongs to the 
	\emph{characteristic variety}, see (\ref{db}) below).
	\item[-b-] The envelope $U(t,x)$ is approximated by the solution of the
	nonlinear Schr\"odinger equation
	\begin{equation}\label{appr1}
	 \dt U+({\bf c_g}\cdot \nabla) U
	-\eps\frac{i}{2} {\mathcal R}(\partial,\partial) U
	=\eps \widetilde{F}(U),
	\qquad U_{\vert_{t=0}}=U^0,
	\end{equation}
	where the group velocity ${\bf c_g}$, the second order differential
operator ${\mathcal R}(\partial,\partial)$ and the nonlinearity
$\widetilde{F}$ can be explicitly given in terms of the data.
\end{enumerate}

The Schr\"odinger approximation (\ref{appr})-(\ref{appr1}) has
been rigorously justified \cite{DJMR,JMR,Lannes} for times of order
$O(1/\eps)$ in the usual
situation where the typical scale for the space variations of $U^0$
is of order $1$. This means that the number of oscillations (or optical
cycles) in the laser pulse is of order $O(1/\eps)$. Recently, however, 
lasers with ultrashort pulses have been developed, for which the number
of optical cycles is much smaller; for such pulses, the Schr\"odinger
equation proves completely inaccurate and various authors proposed other
ways of describing the asymptotics of (\ref{eqd}) when $\eps\to 0$. Some
of these results are briefly recalled below; their common point is that they
all abandon the \emph{SVEA}, because the widely accepted
``practical rule'' 
\begin{equation}\label{thumb}
	\mbox{The SVEA (\ref{appr}) is valid 
	if } \vert \nabla U^0\vert_\infty\ll \frac{1}{\eps}
\end{equation}
is enforced when the pulses get very small.\\
Alterman and Rauch \cite{AR1,AR2,AR3} modeled short pulses by replacing the
fast oscillating term in the initial condition 
by a fast decaying one; more precisely,
they modified the initial condition for (\ref{eqd}) as follows:
\begin{equation}\label{ICAR}
	u_{\vert_{t=0}}=U^0(x)e^{i\frac{{\bf k}\cdot x}{\eps}}+\cc
	\quad \leadsto \quad
	u_{\vert_{t=0}}= U^0(x,\frac{{\bf k}\cdot x}{\eps}),
\end{equation}
with $U^0(x,z)\to 0$ as $z\to \infty$, and the SVEA
(\ref{appr}) is consequently replaced by
$$
	u(t,x)\sim U\big(t,x,\frac{{\bf k}\cdot x-\omega({\bf k})t}{\eps}\big),
$$
with $U(t,x,z)\to 0$ as $z\to\infty$. The Schr\"odinger equation 
(\ref{appr1}) is then replaced by
\begin{equation}\label{AlR}
	\dt\dz U+({\bf c_g}\cdot \nabla) \dz U
	+\eps\frac{1}{2} {\mathcal R}(\partial,\partial) U
	=\eps \dz \widetilde{F}(U),
	\qquad U_{\vert_{t=0}}=U^0;
\end{equation}
this approximation (rigorously justified) uses the fact that the group 
velocity ${\bf c_g}$
does not depend on $\vert {\bf k}\vert$ and is therefore only valid 
in \emph{nondispersive media} ($E=0$ in (\ref{eqd})). 
Alterman and Rauch's approach has been generalized in \cite{CJSW,SW} taking
into account the particularities of the optical susceptibility of
some cubic nonlinear media such as silica, and finally obtaining
a \emph{quasilinear} variant of (\ref{AlR}), which is rigorously
justified in the linear case.\\
In order to model the propagation of ultrashort pulses in
\emph{dispersive} media, Barrailh and Lannes \cite{BL} chose another
approach based on the functional tools developed in \cite{Lannes-sp}, which
consist in replacing the initial condition for (\ref{eqd}) as follows
$$
	u_{\vert_{t=0}}=U^0(x)e^{i\frac{{\bf k}\cdot x}{\eps}}+\cc
	\quad \leadsto \quad
	u_{\vert_{t=0}}= U^0(x,0,\frac{{\bf k}\cdot x}{\eps}),
$$
where the Fourier transform of the initial profile $U^0(x,T,Z)$ with
respect to $T$ and $Z$ is an \emph{$H^s$-valued measure of bounded variation}.
This general framework allows one to consider initial data of
the form (\ref{eqd}) --for which the bounded variation measure is
obviously $U^0(x)\delta_{(\omega({\bf k}),{\bf k})}$-- and of the
form (\ref{ICAR}). A generalization of (\ref{AlR}) is then derived and
rigorously justified; this equation however has the drawback of being
quite complicated because the transport operator $\dt+{\bf c_g}\cdot\nabla$
must be replaced by a nonlocal operator modeling the fact that
the group velocity ${\bf c_g}$ depends on the frequency in dispersive
media.

An important characteristic of ultrashort pulses that we did not mention 
so far is that their frequency spectrum is broad (while for usual wave
packets as the initial condition of (\ref{eqd}), and taking the Fourier 
transform with respect to the fast scale $x/\eps$, it is essentially
contained in a $O(\eps)$ neighborhood of ${\bf k}$; see for instance \cite{BF}). Since
the dispersion relation of the Schr\"odinger equation (\ref{appr1})
is a second  order 
Taylor expansion of the exact dispersion relation of (\ref{eqd}) at
${\bf k}$, the error is quite important if frequencies far from ${\bf k}$
must be taken into account; in addition to the violation of the 
practical rule (\ref{thumb}), this is another reason why the Schr\"odinger
approximation breaks down for short pulses. This phenomenon is not specific
to short pulses since it occurs for all pulses with large frequency spectrum
(typical examples are \emph{chirped pulses}).\\
An alternative way to replace the NLS approximation for such pulses is 
therefore to focus on the dispersive properties of the asymptotic model.
Instead of abandoning the SVEA (\ref{appr}) as in \cite{AR1,AR2,AR3,SW,CJSW,BL}, various authors \cite{BK,CGL} chose to make this approximation but
kept the full dispersive properties of the original equations (\ref{eqd}), thus
avoiding the disastrous (for large spectrum pulses in dispersive media)
second order Taylor expansion of the dispersion
relation. Consequently, the approximations thus obtained
might have a slightly smaller range of validity, but are undoubtly simpler
and moreover provide a more precise approximation ($O(\eps)$ versus $o(1)$).
For instance, the equation derived in \cite{CGL}, and which we call
\emph{full dispersion model} here, reads
$$
	\dt U+\frac{i}{\eps}(\omega({\bf k}+\eps D)-\uom)U
	=\eps \widetilde{F}(U),\qquad
	U(x)=U^0(x),
$$
where $\omega(\cdot)$ parameterizes the graph of the relevant sheet of the
characteristic variety (quite obviously (\ref{appr1}) can be deduced from
this equation by Taylor expanding $\omega(\cdot)$ at ${\bf k}$). This
model is much simpler than the one derived in \cite{BL} and furnishes
very satisfactory results; however it is still a nonlocal equation
and its resolution requires spectral methods. One
of the goals of this paper is to derive a new approximation, with, in practical, the same dispersive 
qualities as the full dispersion model, but keeping the same
level of complexity as the usual Schr\"odinger equation (\ref{appr1}).

\bigbreak

Before describing with more details the results of this paper, let us
introduce here two kinds of initial data for (\ref{eqd}), which
we will often refer to throughout the article:
\begin{itemize}
\item Short-pulses: the initial profile $U^0(x)$ in (\ref{eqd}) is taken
of the form
\begin{equation}\label{IC}
	U^0(x)=f(\frac{x-x_0}{\beta}),
\end{equation}
with $0<\beta\leq 1$ and $f$ a smooth function; 
the case $\beta=1$ corresponds to  classical
laser pulses, and short pulses to $\beta\ll 1$ (see Fig. \ref{figpulse});
the number of optical cycles for such a short pulse is thus $O(\eps/\beta)$;
\begin{figure}
	\includegraphics[width=6cm,angle=270]{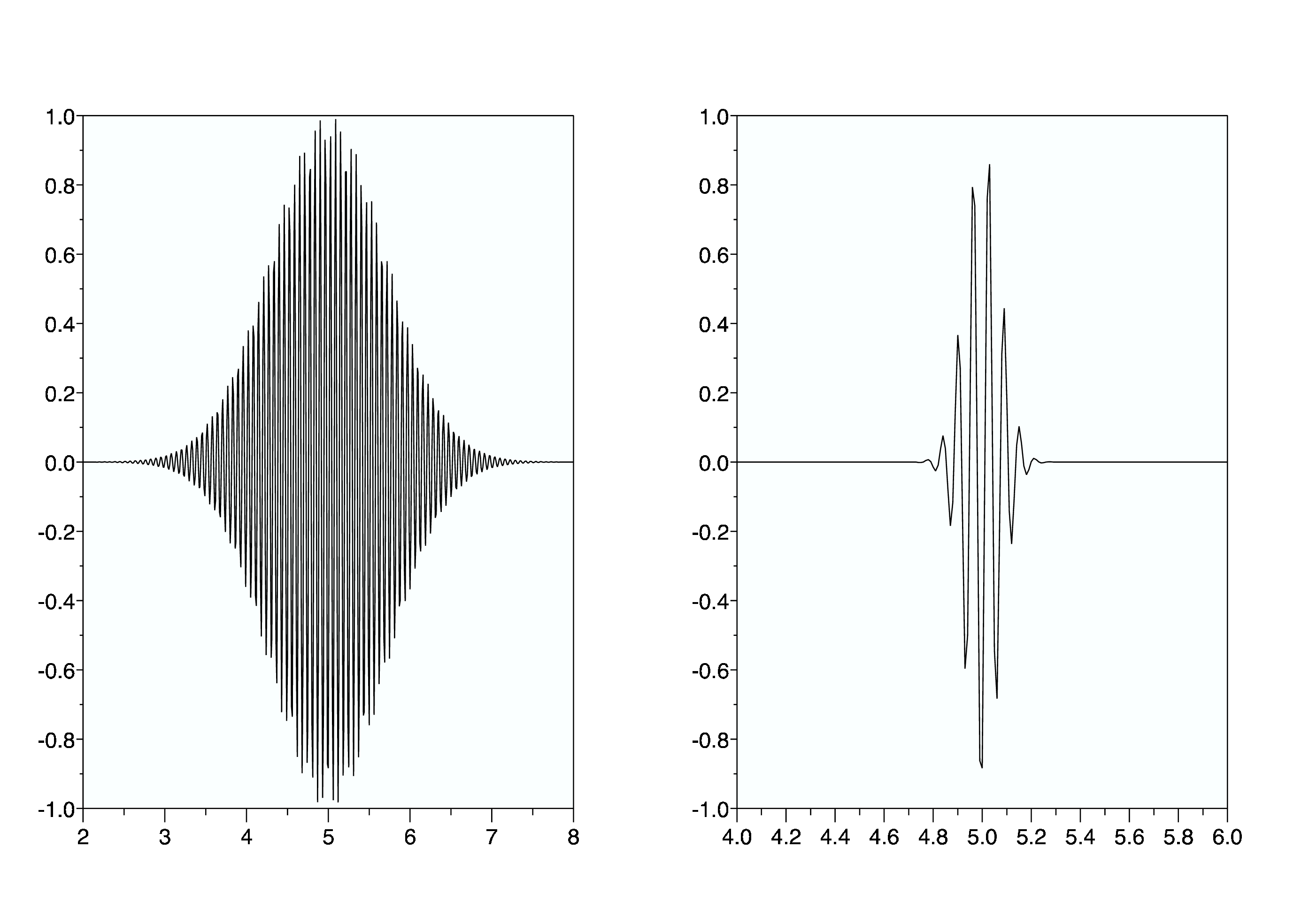}
	\caption{Initial condition 
	$U^0(x)e^{i\frac{{\bf k}\cdot x}{\eps}}+\cc$		
	with $U^0$ as in (\ref{IC}) and $f=e^{-x^2}$, $x_0=5$, $\eps=0.01$,
	and
	for $\beta=1$ and $\beta=0.1$}\label{figpulse}
\end{figure}
\item Chirped pulses: the initial profile $U^0(x)$ in (\ref{eqd}) is taken
of the form
\begin{equation}\label{chirped}
	U^0(x)=f(x-x_0)\cos(\frac{1}{\beta}\cos(\frac{x-x_0}{\beta})),
\end{equation}
with $0< \beta\leq 1$ and $f$ a smooth function; the case
$\beta=1$ corresponds to classical laser pulses, and chirped pulses
to $\beta\ll 1$ (see Fig. \ref{figchirped}).
\begin{figure}
	\includegraphics[width=6cm,angle=270]{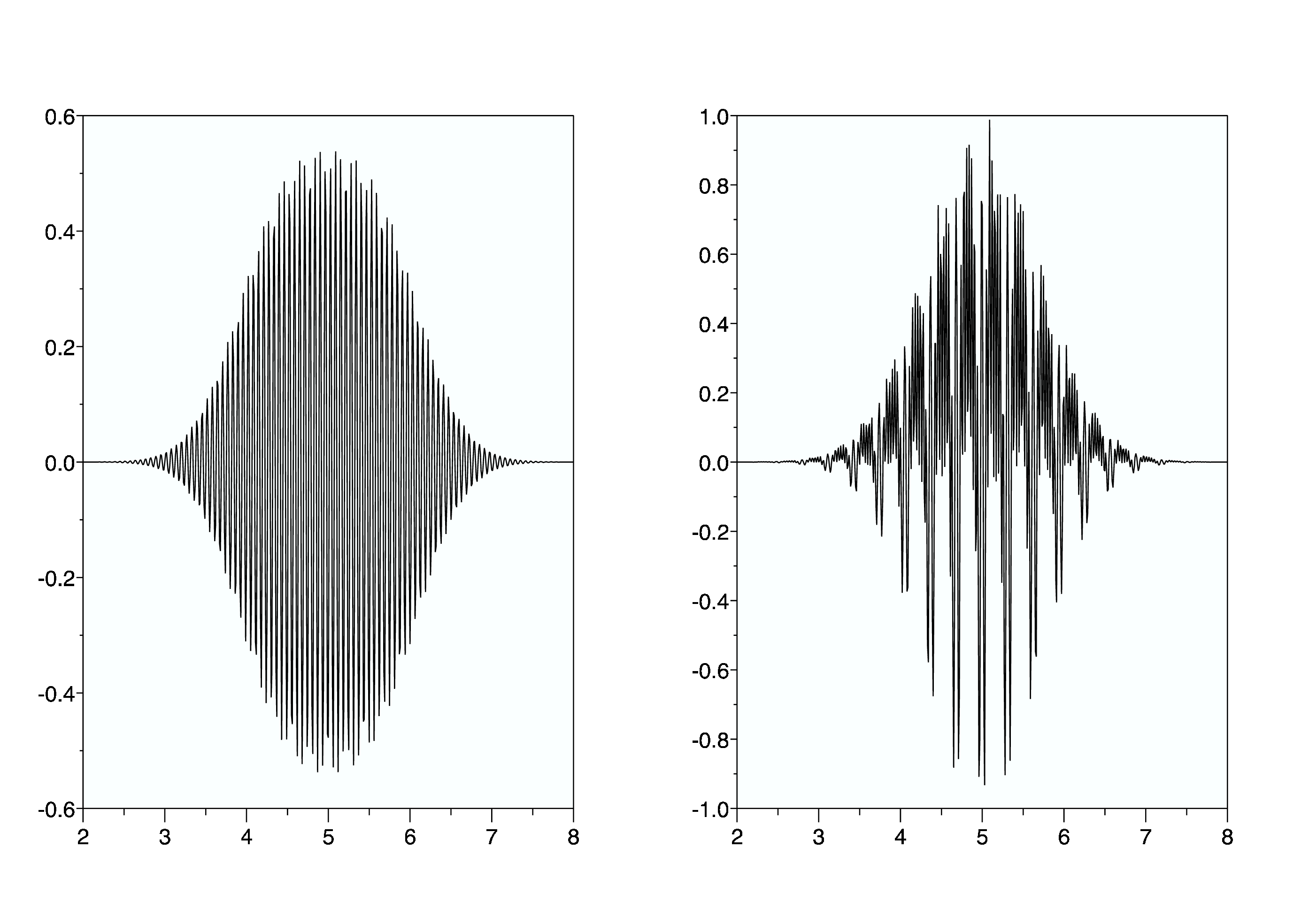}
	\caption{Initial condition 	
	$U^0(x)e^{i\frac{{\bf k}\cdot x}{\eps}}+\cc$
	with $U^0$ as in (\ref{chirped}) and $f=e^{-x^2}$, 
	$x_0=5$, $\eps=0.01$, and
	for $\beta=1$ and $\beta=0.1$}\label{figchirped}
\end{figure}
\end{itemize}

More precisely, we propose here to:
\begin{enumerate}
	\item Provide a framework simpler than \cite{BL} but
	general enough to handle any kind of large spectrum pulses
	(such as ultrashort and chirped pulses). The idea here is
	to keep the SVEA and to work with
	envelopes which are in a Wiener algebra. The reason of this
	choice is that when $\beta\to 0$, the Wiener norm (which controls
	the $L^\infty$-norm)
	of the initial envelopes (\ref{IC})
	remain bounded while 
	any Sobolev norm controlling the $L^\infty$-norm grows to infinity;
	\item Rigorously prove the ``practical rule'' (\ref{thumb}). We
	show that the SVEA (\ref{appr}) 
	makes sense if $\vert \nabla U^0\vert_W\ll \frac{1}{\eps}$,
	where $\vert\cdot\vert_W$ is the Wiener norm (see (\ref{defW}) below); 
	in the case of short-pulses (\ref{IC}), this condition
	is equivalent to  $\eps\ll \beta$;
	\item Establish precise error estimates for the full dispersion and
	Schr\"odinger models. This allows us to give precise estimates on the
	range of validity of these models;
	\item Derive and rigorously justify
	a new family of Schr\"odinger equations with 
	improved frequency dispersion and whose formulation is 
	purely differential (without nonlocal operator). The idea
	is to approximate the nonlocal operator of the full dispersion
	model by a suitable rational function, following an idea
	which proved very useful in water-waves theory
	(derivation of the BBM equation from the KdV equation \cite{BBM}, 
	derivation
	of Boussinesq models with improved frequency dispersion 
	\cite{BCS,BCL}).
	 These equations read
	\begin{equation}\label{eqnew}
	\begin{array}{l}
	\dsp \big(1-i\eps{\bf b}\cdot\nabla-\eps^2\nabla\cdot B\nabla\big)
	\dt U
	+({\bf c_g}\cdot\nabla)U\\
	\indent \dsp -\eps\frac{i}{2}{\mathcal R}(\partial,\partial)U
	+i\eps \nabla\cdot (\nabla\omega_1(\uk){\bf b}^T)\nabla U
	+\eps^2 {\bf C}(\nabla)U
	=\widetilde{F}(U),
	\end{array}
	\end{equation}
	where ${\bf b}\in \C^d$, $B\in {\mathcal M}_{d\times d}(\R)$ and 
	${\bf C}:\C^d\times\C^d\times\C^d\to \C$
	is a trilinear mapping (taking ${\bf b}=0$, $B=0$  and ${\bf C}=0$
	gives therefore (\ref{appr1}));
	\item Present some numerical comparisons between the different
	asymptotic models, for ultrashort and chirped pulses.
\end{enumerate}

\subsection{Organization of the paper}

Section \ref{sectAs} is devoted to the derivation and justification 
of the different asymptotic models. Basic assumptions, definitions, and tools
are first given in \S \ref{sectAs1} and we then proceed to
prove the envelope approximation (\ref{appr}) in \S \ref{sectAs2}; 
in particular, precise error estimates are established which allow
one to give a rigorous proof of the practical rule (\ref{thumb}).
In \S \ref{sectAs3}, we then derive and give precise error estimates 
in a Wiener algebra
setting of the full dispersion model (\S \ref{sectAs31}) and
the Schr\"odinger model (\S \ref{sectAs32}). Finally the new Schr\"odinger
equations (\ref{eqnew}) are derived in \S \ref{sectAs33}.

In Section \ref{sectNum}, we present numerical computations in order
to make a comparison of the different asymptotic models.

\subsection{Notations}

\noindent
- We denote generically by $\cst$ a constant whose value may change
from one line to another;\\
- We use the abbreviation $\cc$ for ``complex conjugate'', so that
for all $a\in\C$, $a+\cc=a+\overline{a}$;\\
- If ${\bf a}, {\bf b}\in \C^d$ then ${\bf a}\cdot {\bf b}$ denotes
the scalar (not hermitian) product: ${\bf a}\cdot {\bf b}=\sum_{j=1}
^d{\bf a}_j{\bf b}_j$.\\
- For all $x\in \R^d$, we write $x=(x_1,\dots,x_d)$, and $\partial_j$
stands for $\partial_j=\partial_{x_j}$; we also write $\nabla=(\partial_1,\dots,\partial_d)^T$ and, for all multi-index $\alpha\in\N^d$, $\partial^\alpha=\partial_1^{\alpha_1}\dots\partial_d^{\alpha_d}$;\\
- For all $j=1,\dots,d$, we write $D_j=\frac{1}{i}\partial_j$ and
$D=(D_1,\dots,D_d)=\frac{1}{i}\nabla^T$;\\
- The Fourier transform of a distribution $u\in {\mathfrak S}(\R^d)$
is denoted either by ${\mathcal F}u$ or $\widehat{u}$;\\
- We use the classical notation for Fourier multipliers: 
$$
	f(D)u={\mathcal F}^{-1} \big(\xi\mapsto f(\xi)\widehat{u}(\xi)\big);
$$
- If $f\in L^\infty(\R^d)$, we simply write $\vert f\vert_\infty=\vert f\vert_{L^\infty}$.

\section{Asymptotic results}\label{sectAs}

This section is devoted to the study of the asymptotic behavior when 
$\eps$ goes to zero of the
solution  to the initial value problem (ivp)
\begin{equation}
	\label{eq1}
	\left\lbrace
	\begin{array}{l}
	\dsp \dt u+A(\partial)u+\frac{E}{\eps}u=\eps T(u,\overline{u},u),
	\vspace{1mm}\\
	\dsp u\init=U^0(x)e^{i\frac{\uk\cdot x}{\eps}}+\cc,
	\end{array}\right.
\end{equation}
where $A(\partial)$, $E$ and  $T(\cdot,\cdot,\cdot)$ are defined in the next 
section.

\subsection{Basic assumptions and tools}\label{sectAs1}

We make the following assumption on the operators $A(\partial)$, $E$
and $T(\cdot,\cdot,\cdot)$ which appear in (\ref{eq1}):
\begin{assu}
	\label{as1}
	{\bf i.} The system (\ref{eq1}) is symmetric hyperbolic in the
	sense that for some $n\geq 1$:
	\begin{itemize}
	\item One has $A(\partial)=\sum_{j=1}^d A_j\partial_j$, and the $A_j$
	are $n\times n$ symmetric, real-valued, matrices;
	\item The $n\times n$ matrix $E$ is real and skew-symmetric;
	\end{itemize}
	{\bf ii.} The mapping
	$$
	T=
	\begin{array}{ccc}
	\C^{3n} &\to & \C^n\\
	(u^1,u^2,u^3)&\mapsto& T(u^1,u^2,u^3),
	\end{array}
	$$
	is linear with respect to $u^1$, $u^2$ and $u^3$.
\end{assu}
\begin{exam}
A standard model for the propagation of a beam in a Kerr medium is the 
Maxwell-Lorentz system which can be written in 
dimensionless form as 
\begin{equation}\label{MX}
	\left\lbrace
	\begin{array}{l}
	\dt B+\curl E=0,\vspace{0.5mm}\\
	\dt E -\curl B+\frac{1}{\eps}Q=0,\vspace{0.5mm}\\
	\dt Q-\frac{1}{\eps}(E-P)=\vert P\vert^2P,\vspace{0.5mm}\\
	\dt P-\frac{1}{\eps}Q=0,
	\end{array}\right.
\end{equation}
where $(E,B)$ is the electromagnetic field, $P$ the polarization and
$Q=\eps\dt P$. This system is of the form (\ref{eq1}) and satisfies
Assumption \ref{as1} with $n=4d$
and
$$
	A(\partial)=
	\left(\begin{array}{cccc}
	0& \nabla\times & 0 & 0\\
	-\nabla\times & 0 &0 &0\\
	0 & 0 & 0 & 0\\
	0 & 0 & 0 & 0
	      \end{array}\right)
	\quad\mbox{ and }\quad
	E=\left(\begin{array}{cccc}	
	0 & 0 & 0 & 0\\
	0 & 0 & I_{d\times d} & 0\\
	0 & -I_{d\times d}& 0 & I_{d\times d}\\
	0 & 0 & -I_{d\times d}& 0
		\end{array}\right)
$$
(the entries in the above matrices are $d\times d$ matrices); denoting
$u^j=(B^j,E^j, Q^j,P^j)$ ($j=1,2,3$), the nonlinearity
is given by
$$
	T(u^1,u^2,u^3)=\left(
	\begin{array}{c}
	0\\ 0\\ (P^1\cdot P^2)P^3\\0
	\end{array}\right).
$$
\end{exam}
\begin{exam}
A simple toy model is the following Klein-Gordon system
\begin{equation}\label{KG}
	\dt u
	+\left(
	\begin{array}{cc}
	0 & \nabla^T\\
	\nabla & 0
	\end{array}\right)u
	+\frac{1}{\eps}
	\left(
	\begin{array}{cc}
	0 & -{\bf v}\\
	{\bf v}^T & 0
	\end{array}\right)u=
	\vert u\vert^2
		\left(
	\begin{array}{cc}
	0 & -{\bf v}\\
	{\bf v}^T & 0
	\end{array}\right)u,
\end{equation}
with $u: \R^+_t\times\R^d_x\mapsto \C^{1+d}$
and ${\bf v}\in \R^d\backslash\{0\}$.\\
Quite obviously, (\ref{KG}) is of the form (\ref{eq1}) and 
satisfies Assumption \ref{as1} with $n=1+d$,
$$
	A(\partial)=
	\left(
	\begin{array}{cc}
	0 & \nabla^T\\ \nabla & 0
	\end{array}
	\right),
	\qquad
	E=	\left(
	\begin{array}{cc}
	0 & -{\bf v}^T\\
	{\bf v} & 0
	\end{array}\right)
$$
and $
	T(u^1,u^2,u^3)=(u^1\cdot u^2)	\left(
	\begin{array}{cc}
	0 & -{\bf v}^T\\
	{\bf v} & 0
	\end{array}\right)u^3$.
\end{exam}

Under Assumption \ref{as1}, the matrix 
$A(\bk)+\frac{E}{i}=\sum_{j=1}^dA_j \bk_j+\frac{E}{i}$ 
is hermitian for all $\bk\in \R^d$, and thus diagonalizable,
with real eigenvalues. We can therefore define the
\emph{characteristic variety} as
\begin{equation}\label{db}
	{\mathcal C}:=\{ (\omega,\bk)\in \R\times\R^d, \omega
	\mbox{ is an eigenvalue of } A(\bk)+\frac{E}{i}\};
\end{equation}
introducing
\begin{equation}
	\label{eq3}
	\cL(\omega,\uk):=-\omega I+A(\bk)+\frac{E}{i},
\end{equation}
one can equivalently define $\cC$ as the set of all 
$(\omega,\bk)\in \R\times\R^d$ such that $\cL(\omega,\bk)$ is \emph{not} invertible.

It is classical and not restrictive for our present concern to make
the following assumption on $\cC$:
\begin{assu}
	\label{as2}
	There exists $m\in\N$ and $m$ different smooth functions 
	$\omega_j\in C^\infty(\R^d\backslash\{0\})$ ($j=1,\dots,m$), 
	such that for all $\bk\in \R^d\backslash\{0\}$, the
	eigenvalues of $A(\bk)+\frac{E}{i}$ are exactly
	$\omega_j(\bk)$ ($j=1,\dots,m$).
\end{assu}

Under this Assumption \ref{as2}, one can write,
for all  $\bk\in \R^d\backslash\{0\}$,
\begin{equation}	
	\label{eq2}
	A(\bk)+\frac{E}{i}=\sum_{j=1}^m\omega_j(\bk)\pi_j(\bk),
\end{equation}
where $\pi_j(\bk)$ denote the eigenprojector associated to $\omega_j(\bk)$
(in particular, $\pi_j\in C^\infty(\R^d\backslash\{0\};{\mathcal M}_n(\C)$).

\bigbreak

We finally need a last assumption on the wave number $\uk$ of the initial
data of the ivp (\ref{eq1}) in order to justify the asymptotic equations
derived in this article.
\begin{assu}
	\label{as3}
	One has $\uk\neq 0$ and, with $\uom=\omega_1(\uk)$:
	\begin{itemize}
	\item One has $(3\uom,3\uk)\notin\cC$;
	\item With the notations of Assumption \ref{as2}, 
	$$
	\exists c_0>0,\quad \forall j=2,\dots,m, \qquad
	\inf_{\bk\in \R^d\backslash\{0\}}
	\big\vert \uom -\omega_j(\bk)\big\vert
	\geq c_0.
	$$
	\end{itemize}
\end{assu}
\begin{rema}
	The first part of the assumption excludes resonances with the
	third harmonic. The results presented here could easily be 
	extended to cover such a situation, but it is not 
	restrictive at all to make this assumption. 
\end{rema}
\begin{exam}
	For the Maxwell equations (\ref{MX}), one can check after some 
	computations that Assumption \ref{as2} is satisfied with $m=7$
	and
	\begin{eqnarray*}
	\omega_1({\bf k})&=&\frac{1}{2}\big(\sqrt{2(1+\vert{\bf k}\vert)+\vert{\bf k}\vert^2}+\sqrt{2(1-\vert{\bf k}\vert)+\vert{\bf k}\vert^2}\big),\\
	\omega_2({\bf k})&=&\sqrt{2},\\
	\omega_3({\bf k})&=&\frac{1}{2}\big(\sqrt{2(1+\vert{\bf k}\vert)+\vert{\bf k}\vert^2}-\sqrt{2(1-\vert{\bf k}\vert)+\vert{\bf k}\vert^2}\big),\\
	\omega_4({\bf k})&=&0,
	\end{eqnarray*}
and $\omega_5=-\omega_3$, $\omega_6=-\omega_2$, $\omega_7=-\omega_1$
	(so that $\omega_1>\omega_2>\dots>\omega_7$ on $\R^d\backslash\{0\}$);
	one can therefore take $c_0=\omega_1(\uk)-\sqrt{2}>0$.
\end{exam}
\begin{exam}
	For the Klein-Gordon system (\ref{KG}) one readily checks that
	Assumption \ref{as2} is satisfied with $m=2$,
	$\omega_1({\bf k})=\sqrt{\vert {\bf k}\vert^2+\vert {\bf v}\vert^2}$ 
	and $\omega_2=-\omega_1$. One can then remark that 
	Assumption \ref{as3} also holds for all $\uk\neq 0$
	and $c_0=\omega_1(\uk)+\vert {\bf v}\vert$. 
\end{exam}

\bigbreak

We finally end this section with some results on the Wiener algebras. First 
recall that $W^k(\R^d;\C^n)$ ($k,n\in\N$) and $W(\R^d_X\times\T_\theta;\C^n)$ (which will be denoted by 
$W(\R^d\times\T;\C^n)$ in the sequel)
are defined as
\begin{equation}\label{defW}
	W^k(\R^d;\C^n):=\{
	f\in {\mathcal S}'(\R^d)^n,\forall \alpha\in\N^d, 
	\vert\alpha\vert\leq k, \vert {\partial^\alpha f}\vert_W<\infty\},
\end{equation}
with $\vert f\vert_W:=\vert \widehat{f}\vert_{L^1(\R^d;\C^n)}$, 
and 
$$
	W(\R^d\times \T;\C^n):=\{
	f=\sum_{n\in\Z} f_n(X)e^{in\theta},
	\vert f\vert_{W(\R^d\times\T)}:=\sum_n\vert f_n\vert_W<\infty\}
$$
(when $k=0$, we write $W(\R^d;\C^n)$ instead of $W^0(\R^d;\C^n)$).\\
The classical properties of the Wiener algebras used in this article
are recalled in the following proposition.
\begin{prop}
	\label{pr1}
	{\bf i.} The space $W^k(\R^d;\C)$, $k\in\N$, 
	(resp. $W(\R^d\times\T;\C)$) 
	is an algebra
	in the sense that the mapping 
	$(f,g)\mapsto fg$ is continuous
	from $W^k(\R^d;\C)^2$ into $W^k(\R^d;\C)$ 
	(resp. $W(\R^d\times \T;\C)^2$ into $W(\R^d\times \T;\C)$).\\
	{\bf ii.} If $M$ is a skew-symmetric, real valued $n\times n$
	matrix, then $\exp(-iM)$ is unitary on
	$W^k(\R^d;\C^n)$ (resp. $W^k(\R^d\times \T;\C^n)$).\\
	{\bf iii.} If $f\in W(\R^d;\C)$ then for all $\beta>0$,
	one has $f(\frac{\cdot}{\beta})\in W(\R^d;\C)$ and
	$\vert f(\frac{\cdot}{\beta})\vert_W=\vert f\vert_W$.
\end{prop}
\begin{rema}
As said in the introduction, the third point of Proposition \ref{pr1}
is the main motivation to work with Wiener algebra rather than
Sobolev spaces, because the $W(\R^d;\C^n)$-norm of initial conditions
of the form (\ref{IC}) remains bounded (constant) while its $H^s$-norm is
of size $O(\beta^{d/2-s})$ as $\beta\to 0$. This framework is somehow
a simplified version of the functional setting of \cite{BL} (where the
SVEA is not made) and has also proved useful in the study of wave-packets
\cite{BF2} or localized solutions \cite{PS}.
\end{rema}
\subsection{The envelope approximation}\label{sectAs2}

This section is devoted to the justification of the envelope approximation
which states that the exact solution to (\ref{eq1}) can be described at
leading order by $u_{app}^\eps$ defined as
\begin{equation}
	\label{uapp}
	u_{app}^\eps(t,x)= U(t,x)e^{i\frac{\uk\cdot x-\uom t}{\eps}}+\cc,
\end{equation}
where the \emph{envelope} $U$ solves the \emph{envelope equation}
\begin{equation}
	\label{eq4}
	\left\lbrace
	\begin{array}{l}
	\dt U+\frac{i}{\eps}\cL(\uom,\uk+\eps D)U
	=\eps \cT(U)\\
	U\init(x)=U^0(x),
	\end{array}\right.
\end{equation}
with $\cL(\cdot,\cdot)$ given by (\ref{eq3}) and $\cT$ is defined 
as
$$
	\cT(U)=T(\bU,U,U)+T(U,\bU,U)+T(U,U,\bU).
$$

The interest of the envelope equation (\ref{eq4}) with respect to the
original ivp (\ref{eq1}) is that the fast oscillating scale has been
removed from the initial data. The main result of this section is the
following theorem.
\begin{theo}\label{th1}
	Let Assumptions \ref{as1}, \ref{as2} and \ref{as3} be satisfied,
	and let $a,b \in W^1(\R^d;\C^n)$, and 
	$$
	U^0=\pi_1(\uk)a+\eps b
	\qquad \mbox{ (Polarization condition)}.
	$$
	{\bf i.} There exists a time $\tau_0>0$ such that for all $0<\eps<1$,
	there is a unique 
	solution $U\in C([0,\frac{\tau_0}{\eps}); W(\R^d;\C^n))$ to
	(\ref{eq4});\\
	{\bf ii.} For all $0<\tau<\tau_0$, there exists
	$\eps_0$ such that for all $0<\eps<\eps_0$, there is a unique
	solution $u_{ex}^\eps\in C([0,\tau/\eps]\times \R^d)^n$ to (\ref{eq1}),
	and one has   
	$$
	\vert u_{ex}^\eps -u^\eps_{app}
	\vert_{L^\infty([0,\tau/\eps]\times \R^d)^n}
	\leq \eps C(\tau,\vert U^0\vert_W)(1+\vert b\vert_W
	+\vert \nabla U^ 0\vert_W),
	$$
	where $u^\eps_{app}$ is as defined in (\ref{uapp}).
\end{theo}
\begin{rema}\label{rmth}
	Assuming that $\vert U^0\vert_W$ and $\vert b\vert_W$ are
	$O(1)$ quantities --which is of course the case for short pulses--,
	one deduces the following result
	$$
	\mbox{The SVEA is valid if } \vert\nabla U^0\vert_W\ll \frac{1}{\eps},
	$$
	and Theorem \ref{th1} thus provides a rigorous basis for
	the ``practical rule'' (\ref{thumb}). 
	When working with short pulses with initial condition
	(\ref{IC}), it is easy to check that this
	condition reads simply $\eps\ll \beta$. 
\end{rema}
\begin{rema}\label{rmth2}
	Working in the more classical framework of Sobolev spaces,
	one could establish an error estimate similar to the one
	given by the theorem, but with $H^s$-norms ($s>d/2$) instead
	of $W$-norms in the rhs of the estimates. For
	short pulses with initial data (\ref{IC}), the control would
	therefore be of the form 
	$\eps C(\tau,\frac{1}{\beta^{s-d/2}})\frac{1}{\beta^{s+(1-d)/2}}$,
	which is obviously useless when $\beta\to 0$.
\end{rema}
\begin{proof}
Let us prove the following lemma, which implies the first point of the
theorem.
\begin{lemm}
	\label{lm1}
	Let $U^0\in W(\R^d;\C^n)$, and assume 
	that Assumption \ref{as1} 
	is satisfied.\\
	There exists 
	$\tau_0>0$ such that for all $0<\eps<1$ one has a unique 
	solution $U\in C([0,\frac{\tau_0}{\eps}); W(\R^d;\C^n))$
	to (\ref{eq4}). For
	all $0<\tau <\tau_0$, one has
	$$
	(a)\qquad \sup_{0\leq t\leq \tau /\eps} \vert U(t)\vert_W
	\leq C(\tau ,\vert U^0\vert_W);
	$$
	if moreover $U^0\in W^1(\R^d;\C^n)$ then one also has
	$$
	(b) \qquad \sup_{0\leq t\leq \tau /\eps} \vert \nabla U(t)\vert_W
	\leq C(\tau ,\vert U^0\vert_W)\vert \nabla U^0\vert_W.
	$$
\end{lemm}
\begin{proof}
Uniqueness is obvious, and to prove existence, we use a classical
iterative method: let $U_0=U^0$ and, for all $n\in \N$,
$$
	U_{n+1}(t)=S(t)U^0+\eps\int_0^t S(t-t')
	\cT(U_n) dt',
$$
with $S(t)=\exp(-\frac{it}{\eps}\cL(\uom,\uk+\eps D))$.\\
Since $\cL(\uom,\uk+\eps\xi)$ is real and skew-symmetric for 
all $\xi\in\R^d$, we can use Proposition \ref{pr1} 
and the trilinearity of $T$ to get
$$
	\sup_{[0,t]}\vert U_{n+1}\vert_W\leq \vert U^0\vert_W+
	\eps \cst t \big(\sup_{[0,t]}\vert U_n\vert_W\big)^3,
$$
and
$$
	\sup_{[0,t]}\vert U_{n+1}-U_n\vert_W\leq 
	\eps \cst t C\big(\sup_{[0,t]}\vert U_n\vert_W,
	\sup_{[0,t]}\vert U_{n-1}\vert_W\big)
	\sup_{[0,t]}\vert U_n-U_{n-1}\vert_W,
$$
and it follows easily that the sequence $(U_n)_n$ converges in
$C([0,\frac{\tau }{\eps}];W^0(\R^d;\C^n))$ (for some $\tau >0$)
 to a solution $U$ of
\begin{equation}
	\label{eq5}
	U(t)=S(t)U^0+\eps\int_0^t S(t-t')
	\cT(U) dt';
\end{equation}
the solution can then be extended to a maximal time interval 
$[0,\frac{\tau _{max}(\eps)}{\eps})$ and one can show with classical 
arguments that $\tau_0:=\inf_{0<\eps<1}\tau_{\max}(\eps)>0$. 
The estimate (a) 
of the lemma also 
follows easily from (\ref{eq5}) and a Gronwall-type lemma.\\
Differentiating (\ref{eq5}) with respect to $x_j$ ($j=1,\dots,d$) and
using Proposition \ref{pr1} to control the $W(\R^d;\C^n)$-norm, one gets
$$
	\vert \partial_j U(t)\vert_W\leq
	\vert \partial_j U^0\vert_W
	+\eps\cst 
	\int_0^t \vert U(t')\vert_W^2\vert \partial_j U(t')\vert_W dt',
$$
and the estimate (b) follows from Gronwall's lemma and the
estimate (a).
\end{proof}

Before going further in the proof of the theorem, let us introduce 
some notation. We decompose the solution $U$ of (\ref{eq4}) provided
by Lemma \ref{lm1} as
$$
	U=U_1+\dots+U_m,
	\quad\mbox{ with }\quad
	U_j=\pi_j(\uk+\eps D)U,
$$
and we also write $U_{II}=U_2+\dots+U_m$.\\
The first step of the proof of the second part of the theorem consists
in controlling $\dt U_1$ uniformly in $\eps$ --which is much better than
the $O(1/\eps)$ estimate on $\dt U$ one can deduce directly from the equation
(\ref{eq4}).
\begin{lemm}
	\label{lm2}
	If Assumptions \ref{as1} and \ref{as2} are satisfied, then
	for all $0<\tau <\tau_0$,
	$$
	\sup_{0\leq t\leq \tau /\eps}\vert \dt U_1\vert_W\leq 
	C(\tau ,\vert U^0\vert_W)\big(1+\vert\nabla U^0\vert_W\big).
	$$
\end{lemm}
\begin{proof}
Multiplying the envelope equation (\ref{eq4}) by 
$\pi_1(\uk+\eps D)$, one gets
\begin{equation}
	\label{eq6}
	\dt U_1+\frac{i}{\eps}\big(\omega_1(\uk+\eps D)-\uom\big)U_1
	=\eps \pi_1(\uk+\eps D)\cT(U).
\end{equation}
Recalling that $\uom=\omega_1(\uk)$, a first order Taylor expansion
shows that
$$
	\big\vert \big(\omega_1(\uk+\eps D)-\uom\big)U_1\big\vert_W
	\leq \eps\vert \nabla  \omega_1\vert_{L^\infty}\vert \nabla U_1\vert_W
$$
(note that one infers $\vert \nabla\omega_1\vert_{L^\infty}<\infty$ from the
observation that for all ${\bf k}\neq 0$, 
$\partial_j\omega_1({\bf k})\pi_1({\bf k})=\pi_1({\bf k})A_j\pi_1({\bf k})$).
It follows therefore from (\ref{eq6}), the trilinearity of $T$, and
Proposition \ref{pr1} that
$$
	\vert \dt U_1\vert_W\leq \vert \nabla \omega_1\vert_{L^\infty}
	\vert \nabla U_1\vert_W +\cst \eps \vert U\vert_W^3,
$$
and the result follows from Lemma \ref{lm1}.
\end{proof}

We now prove that the components $U_j$ ($j\geq 2$) remain
of size $O(\eps)$ if this is initially the case.
\begin{lemm}\label{lm3}
	If Assumptions \ref{as1}, \ref{as2} and \ref{as3} are satisfied,
	and if $U^0=\pi_1(\uk)a+\eps b$, then one
	has, for all $0<\tau <\tau_0$,
	$$
	\sup_{t\in [0,\tau/\eps ]}\vert U_{II}(t)\vert_W
	\leq \eps C(\tau ,\vert U^0\vert_W)(1+\vert b\vert_W
	+\vert \nabla U^0\vert_W).
	$$
\end{lemm}
\begin{proof}
Multiplying (\ref{eq4}) by $\pi_j(\uk+\eps D)$ ($j\geq 2$) gives
\begin{eqnarray*}
	\lefteqn{\dt U_j+\frac{i}{\eps}\big(\omega_j(\uk+\eps D)-\uom\big)U_j
	=
	\eps \pi_j(\uk+\eps D)\cT(U)}\\
	&=& \eps \pi_j(\uk+\eps D)\cT(U_1)
	+\eps \pi_j(\uk+\eps D)\big(\cT(U)-\cT(U_1)\big).
\end{eqnarray*}
With $S_j(t)=\exp\big(-i\frac{t}{\eps}(\omega_j(\uk+\eps D)-\uom)\big)$,
one gets therefore
\begin{eqnarray}
	\nonumber
	U_j(t)&=& S_j(t)U_j^0+\eps\int_0^t
	S_j(t-t')\pi_j(\uk+\eps D)\cT(U_1)dt'\\
	\label{eq6b}
	& &+ \eps\int_0^t
	S_j(t-t')\pi_j(\uk+\eps D)\big(\cT(U)-\cT(U_1)\big)dt'.
\end{eqnarray}
We now  bound the 
$W$-norm of the three terms of the r.h.s. of
(\ref{eq6b}):
\begin{itemize}
\item Estimate of $S_j(t)U_j^0$. Since $S_j(t)$ is unitary on $W(\R^d;\C^n)$, 
one has
$$
	\vert S_j(t)U_j^0\vert_W=\vert U_j^0\vert_W
	=
	\vert \pi_j(\uk+\eps D)U^0\vert_W.
$$
Since moreover one can write
$$
	\pi_j(\uk+\eps D)U^0=(\pi_j(\uk+\eps D)-\pi_j(\uk))U^0+\pi_j(\uk)U^0,
$$
it follows from the
 orthogonality of the projectors $\pi_j$ ($j=1,\dots,m$) that
$$
	\pi_j(\uk+\eps D)U^0=(\pi_j(\uk+\eps D)-\pi_j(\uk))U^0+
	\eps\pi_j(\uk)b.
$$
Since the derivatives of $\pi_j(\cdot)$ are in general not bounded near the
origin, we cannot control the first term of the rhs by a Taylor expansion
and we thus write
\begin{eqnarray*}
	\pi_j(\uk+\eps D)U^0&=&(\pi_j(\uk+\eps D)-\pi_j(\uk))1_{\{\eps\vert D\vert\leq \vert\uk\vert/2\}}U^0\\
	&+&
	(\pi_j(\uk+\eps D)-\pi_j(\uk))1_{\{\eps\vert D\vert\geq \vert\uk\vert/2\}}U^0+
	\eps\pi_j(\uk)b,
\end{eqnarray*}
where $1_{\{\eps\vert\xi\vert\leq \vert\uk\vert/2\}}=1$ if $\eps\vert\xi\vert\leq \vert\uk\vert/2$ and $0$ otherwise.\\
Using the fact that $\pi_j(\cdot)$ is $C^\infty$ on the ball of center
$\uk$ and radius $\vert\uk\vert/2$, we can bound the first term
of the rhs in $W(\R^d,\C^n)$-norm by $\eps\cst \vert \nabla U^0\vert_W$; one
can also check that a similar estimate holds for the second term of the
lhs since one has $1\leq \frac{2\eps}{\vert\uk\vert}\vert \xi\vert$ for all
$\eps\vert\xi\vert\geq \vert\uk\vert/2$. We can thus conclude that
\begin{equation}
	\label{eq0est}
		\vert S_j(t)U_j^0\vert_W
	\leq \eps\cst (\vert b\vert_W+\vert \nabla U^0\vert_W).
\end{equation}
\item Estimate of 
$A:=\eps\int_0^tS_j(t-t')\pi_j(\uk+\eps D)\cT(U_1)dt'$.
Taking the Fourier transform of this term and integrating by parts
yields
\begin{eqnarray*}
	\eps\lefteqn{\int_0^t \exp\big(-i\frac{t-t'}{\eps}(\omega_j(\uk+\eps\xi)-\uom)\big)
	\pi_j(\uk+\eps\xi)\widehat{\cT(U_1)}dt'}\\
	&=&-i\eps\int_0^t \frac{\eps\exp\big(-i\frac{t-t'}{\eps}(\omega_j(\uk+\eps\xi)-\uom)\big)}{\omega_j(\uk+\eps\xi)-\uom}
	\pi_j(\uk+\eps\xi)\dt\widehat{\cT(U_1)}dt'\\
	& &+i\eps\Big[\frac{\eps\exp\big(-i\frac{t-t'}{\eps}(\omega_j(\uk+\eps\xi)-\uom)\big)}{\omega_j(\uk+\eps\xi)-\uom}
	\pi_j(\uk+\eps\xi)\widehat{\cT(U_1)}\Big]_0^t.
\end{eqnarray*}
One deduces therefore, using Assumption \ref{as3}, that
$$
	\sup_{t\in[0,\tau /\eps]}\vert A(t)\vert_W
	\leq\cst \frac{\eps \tau }{c_0}
	\sup_{[0,\tau /\eps]}\vert U_1\vert_W^2\sup_{[0,\tau /\eps]}
	\vert \dt U_1\vert_W
	+\cst \frac{\eps^2}{c_0}\sup_{[0,\tau /\eps]}\vert U_1\vert_W^3,
$$
so that, owing to Lemmas \ref{lm1} and \ref{lm2},
\begin{equation}
	\label{eq1est}
	\sup_{t\in[0,\tau /\eps]}\vert A(t)\vert_W
	\leq
	\eps C(\tau ,\vert U^0\vert_W)(1+\vert\nabla U^0\vert_W).
\end{equation}
\item Estimate of 
$B:=\eps\int_0^t
S_j(t-t')\pi_j(\uk+\eps D)\big(\cT(U)-\cT(U_1)\big)dt'$.
First remark that owing to the trilinearity of $T$, one has for all
$t\in [0,\tau /\eps]$,
$$
	\big\vert \cT(U)(t)-\cT(U_1)(t)\big\vert_W
	\leq
	\cst \sup_{[0,\tau /\eps]} \vert U\vert_W^2 
	\vert U_{II}(t)\vert_W;
$$
using Lemma \ref{lm1}, we obtain therefore 
\begin{equation}
	\label{eq2est}
	\sup_{t\in[0,\tau /\eps]}\vert B(t)\vert_W
	\leq \eps C(\tau , \vert U^0\vert_W)
	\int_0^t \vert U_{II}(t')\vert_Wdt'.
\end{equation}
\end{itemize}
It is now a direct consequence of (\ref{eq6b})$_j$ ($j=2,\dots m$) 
and (\ref{eq0est})-(\ref{eq2est})
that for all $t\in [0,\tau /\eps]$,
$$
	\vert U_{II}(t)\vert_W\leq 
	\eps (\vert b \vert_W+
	C(\tau ,\vert U^0\vert_W)(1+\vert\nabla U^0\vert_W)
	+\eps C(\tau , \vert U^0\vert_W)\int_0^t \vert U_{II}(t')\vert_Wdt',
$$
and the result follows therefore from Gronwall's lemma.
\end{proof}
We are now set to conclude the proof of the theorem. We look for 
an exact solution $u_{ex}^\eps$
to (\ref{eq1}) under the form
$$
	u_{ex}^\eps(t,x)=U_{ex}(t,x,\frac{\uk\cdot x-\uom t}{\eps}),
$$
with $U_{ex}\in W(\R^d\times\T;\C^n)$ itself of the form
$$
	U_{ex}(t,x,\theta)=U_{app}(t,x,\theta)+\eps V(t,x,\theta),
$$
with $U_{app}(t,x,\theta)=U(t,x)e^{i\theta}+\cc$ and $V$ bounded in 
$W(\R^d\times \T;\C^n)$. With $U$ as given by Lemma \ref{lm1}, the equation
that $V$ must solve is
\begin{eqnarray*}
	\lefteqn{\dt V+\frac{i}{\eps}\cL(\uom D_\theta,\uk D_\theta+\eps D)V
	=T(U,U,U)e^{i3\theta}+\cc}\\
	& &+
	\big(T(U_{app}+\eps V,\overline{U_{app}+\eps V},U_{app}+\eps V)
	-T(U_{app},\overline{U_{app}},U_{app})\big).
\end{eqnarray*}
Owing to the first part of 
Assumption \ref{as3}, we can look for $V$ under the form
$$
	V(t,x,\theta)=V_0(t,x,\theta)+\eps V_1(t,x)e^{i3\theta}+\cc,
$$
with $V_1=-i\cL(3\uom,3\uk)^{-1}T(U_1,U_1,U_1)$; the resulting
equation on $V_0$ is 
\begin{equation}\label{V0}
	\dt V_0
	+\frac{i}{\eps}\cL(\uom D_\theta,\uk D_\theta+\eps D)V_0
	=I_1+I_2+I_3,
\end{equation}
with 
\begin{eqnarray*}
	I_1&=&\big(T(U,U,U)-T(U_1,U_1,U_1)\big)e^{i3\theta}+\cc\\
	I_2&=&-\eps \big(\dt+A(\partial)\big)V_1 e^{3i\theta}+\cc\\
	I_3&=&\big(T(U_{app}+\eps V,\overline{U_{app}+\eps V},U_{app}+\eps V)
	-T(U_{app},\overline{U_{app}},U_{app})\big).
\end{eqnarray*}
Let us now bound $I_j$ ($j=1,2,3$) in $W(\R^d\times\T;\C^n)$ and
for all $t\in [0,\tau/\eps]$:
\begin{itemize}
\item From Lemmas \ref{lm1} and \ref{lm3}, one gets
\begin{equation}
	\label{eqsing1}
	\vert I_1(t)
	\vert_{W(\R^d\times \T)}
	\leq
	\eps C(\tau,\vert U^0\vert_W)
	(1+\vert b\vert_W+\vert\nabla U^0\vert_W).
\end{equation}
\item From the definition of $V_1$ and Lemmas \ref{lm1}-\ref{lm2}, 
one has directly
\begin{equation}
	\label{eqsing2}
	\vert I_2(t)\vert_{W(\R^d\times \T)}
	\leq 
	\eps C(\tau,\vert U^0\vert_W)(1+\vert\nabla U^0\vert_W).
\end{equation}
\item From the trilinearity of $T$ and Lemma \ref{lm1}, one gets
\begin{equation}
	\label{eqsing3}
	\vert I_3(t)\vert_{W(\R^d\times \T)}
	\leq 
	\eps C(\tau,\vert U^0\vert_W)(1+\vert V_0(t)\vert_W
	+\eps\vert V_0(t)\vert_W^2+\eps^2\vert V_0(t)\vert_W^3).
\end{equation}
\end{itemize}
By Proposition \ref{pr1}, the semigroup 
$S(t)=\exp(-i\frac{t}{\eps}\cL(\uom D_\theta,\uk D_\theta+\eps D) )$
is unitary on $W(\R^d\times \T)$, so that the estimates
(\ref{eqsing1})-(\ref{eqsing3}) allow one to conclude to the existence 
of a solution $V_0\in C([0,\tau/\eps];W(\R^d\times \T)^n)$ to (\ref{V0})
using a fixed point formulation similar to the one used in the proof
of Lemma \ref{lm1}. After a Gronwall argument, one also gets
\begin{equation}
	\label{eqsing4}
	\sup_{0\leq t\leq \tau/\eps}\vert V_0(t)\vert_W\leq
	C(\tau,\vert U^0\vert_W)(1+\vert b\vert_W+\vert\nabla U^0\vert_W).
\end{equation}
\begin{rema}
	In fact, one finds  that $V_0$ exists a priori on a time
	interval $[0,\tau'/\eps]$, with  $\tau'\leq \tau$.
	However, by a classical procedure of continuous induction
	(such as in the proof of Theorem 3 in \cite{LannesRauch} for instance),
	one can get $\tau'=\tau$, provided that $0<\eps<\eps_0$ with 
	$\eps_0$ small enough.
\end{rema}
It follows from the above that 
$$
	\sup_{t\in [0,\tau/\eps]}
	\vert U_{ex}(t)-U_{app}(t)\vert_{W(\R^d\times\T)}
	\leq
	\eps C(\tau,\vert U^0\vert_W)(1+\vert b\vert_W+\vert\nabla U^0\vert_W),
$$
and the theorem follows therefore from the observation that
$$
	\vert u^\eps_{ex}-u^\eps_{app}\vert_{L^\infty([0,\tau/\eps]\times\R^d)}
\leq \sup_{t\in [0,\tau/\eps]}
	\vert U_{ex}(t)-U_{app}(t)\vert_{W(\R^d\times\T)}.
$$
\end{proof}
\subsection{Approximations by scalar equations}\label{sectAs3}

\subsubsection{The full dispersion model}\label{sectAs31}

The full dispersion model consists in approximating the exact solution
to (\ref{eq1}) by $u^\eps_{app,1}$ defined as
\begin{equation}
	\label{uapp1}
	u_{app,1}^\eps(t,x)= 
	U_{(1)}(t,x)e^{i\frac{\uk\cdot x-\uom t}{\eps}}+\cc,
\end{equation}
where $U_{(1)}$ solves the \emph{full dispersion} scalar equation
\begin{equation}
	\label{eqfull}
	\left\lbrace
	\begin{array}{l}
	\dt U_{(1)}+\frac{i}{\eps}(\omega_1(\uk+\eps D)-\uom)U_{(1)}
	=\eps \pi_1(\uk)\cT(U_{(1)})\\
	U_{(1)}\,\init(x)=U^0(x)
	\end{array}\right.
\end{equation}
and with $\omega_1(\cdot)$ as in Assumption \ref{as1}.\\
The following corollary shows that the full dispersion scalar
equation yields an approximation of same precision than
the envelope equation for times $t\in [0,\tau/\eps]$.
\begin{coro}[Full dispersion model]\label{corofull}
	Under the assumptions of Theorem \ref{th1}, 
	and for all $0<\eps<\eps_0$ ($\eps_0>0$ small enough), there exists
	a unique 
	solution $U_{(1)}\in C([0,\tau_0/\eps);W(\R^d;\C^n))$
	to (\ref{eqfull}).\\
	 For all $0<\tau<\tau_{0}$, one also has
	$$
	\vert u_{ex}^\eps -u^\eps_{app,1}
	\vert_{L^\infty([0,\tau/\eps]\times \R^d)}
	\leq \eps C(\tau,\vert U^0\vert_W)(1+\vert b\vert_W
	+\vert \nabla U^0\vert_W),
	$$
	where $u^\eps_{app,1}$ is as defined in (\ref{uapp1}).
\end{coro}
\begin{rema}
	The quantity $U_{(1)}$ remains $\C^n$-valued, but we call
	(\ref{eqfull}) a \emph{scalar} approximation because the
	operator $\frac{i}{\eps}(\omega_1(\uk+\eps D)-\uom)$ is scalar,
	which is not the case of $\frac{i}{\eps}{\mathcal L}(\uom,\uk+\eps D)$
	in the envelope equation (\ref{eq4}). 	The interest of the 
	FD model is that 
	$\frac{i}{\eps}(\omega_1(\uk+\eps D)-\uom)u$ remains bounded for
	spectrally localized functions $u$, while  
	$\frac{i}{\eps}{\mathcal L}(\uom,\uk+\eps D)u$ is of order $O(1/\eps)$.
	The fast oscillations of the nonpolarized modes must
	therefore be taken into account with the 
	envelope approximation, and the discretization step must therefore
	be much smaller in numerical computations than for the FD model. 
\end{rema}
\begin{rema}
	{\bf i.} Performing the same analysis as in Remark \ref{rmth},
	one can check that the ``practical rule'' also applies for the
	FD model.\\
	{\bf ii.} The FD model has been derived and studied in \cite{CGL} (it is
	called ``intermediate model'' in that reference) for the study
	of chirped pulses with initial data (\ref{chirped}). However
	the error estimate provided in \cite{CGL} is of the 
	form $\eps C(\tau,\vert U^0\vert_{H^s})$, with $s$ large,
	which, as explained in Remark \ref{rmth2} is not small
	for short pulses (or more generally large spectrum 
	--including chirped-- pulses). Note that a variant of the FD
	model with the orthoganal projector $\pi_1(\uk)$ replaced
	by $\pi_1(\uk+\eps D)$ in front of the nonlinearity is also
	studied in \cite{CGL} and gives very good results.
\end{rema}
\begin{proof}
We omit the existence/uniqueness part of the corollary, since it is
obtained with the same tools as for Theorem \ref{th1} (in particular,
taking a smaller $\eps_0$ if necessary, the existence time of the
envelope equation is larger than the existence time for (\ref{eqfull})
and we can thus take the same $\tau_0$ as in Theorem \ref{th1}),
and we thus focus on the error estimate.\\
Denoting as in the proof of Theorem \ref{th1} $U_1=\pi_1(\uk+\eps D)U$,
where $U$ is the solution of the envelope equation, one gets 
from Lemma \ref{lm3} that 
$$
	\sup_{t\in [0,\tau/\eps]}\vert U(t)-U_1(t)\vert_W
	\leq \eps C(\tau,\vert U^0\vert_W)(1+\vert b\vert_W+	
	\vert \nabla U^0\vert_W),
$$
so that it suffices to control $\vert U_1(t)-U_{(1)}(t)\vert_W$ to
prove the corollary. Applying $\pi_1(\uk+\eps D)$ to (\ref{eq4}), one gets
$$
	\dt U_{1}+\frac{i}{\eps}(\omega_1(\uk+\eps D)-\uom)U_{1}
	=\eps \pi_1(\uk+\eps D)\cT(U),
$$
so that the difference $V=U_1-U_{(1)}$ solves
\begin{equation}
	\label{errfull}
	\left\lbrace
	\begin{array}{l}
	\dt V+\frac{i}{\eps}(\omega_1(\uk+\eps D)-\uom)V
	=\eps \pi_1(\uk+\eps D)\cT(U)-\eps\pi_1(\uk)\cT(U_{(1)})\\
	V\init(x)=\pi_1(\uk+\eps D)U^0-U^0.
	\end{array}\right.
\end{equation}
Remark now that
\begin{eqnarray}
	\nonumber
	\lefteqn{\pi_1(\uk+\eps D)\cT(U)-\pi_1(\uk)\cT(U_{(1)})
	=\big(\pi_1(\uk+\eps D)-\pi_1(\uk)\big)\cT(U)}\\
	\label{raj1}
	& &+\pi_1(\uk)	
	\big(\cT(U)-\cT(U_1)\big)
	+\pi_1(\uk)\big(\cT(U_1)-\cT(U_{(1)})\big).
\end{eqnarray}
Since $\vert (\pi_1(\uk+\eps D)-\pi_1(\uk))\cT(U)\vert_W
\leq \eps \cst \vert \nabla \cT(U)\vert_W$ (see the proof of
(\ref{eq0est})), one can use Lemma \ref{lm1}
to bound the first component of the rhs of (\ref{raj1}) from above by
$\eps C(\tau,\vert U^0\vert_W)\vert \nabla U^0\vert_W$. The second component 
of (\ref{raj1}) can be estimated exactly as the term
$I_1$ in (\ref{eqsing1}), while the last one is bounded from above
in $W(\R^d;\C^n)$ by $C(\vert U_1\vert_W,\vert U_{(1)}\vert_W)\vert V\vert_W$.
Since moreover $\vert U_1\vert_W$ is controlled by Lemma \ref{lm1} and
that a similar estimates also holds obviously for $\vert U_{(1)}\vert_W$,
one deduces that for all $0\leq t\leq\tau/\eps$,
$$
	\vert \cT(U(t))-\cT(U_{(1)}(t))\vert_W
	\leq
	\eps C(\tau,\vert U^0\vert_W)(1+\vert b\vert_W+\vert\nabla U^0\vert_W)
	+C(\tau,\vert U^0\vert_W)\vert V(t)\vert_W.
$$
This inequality, together with an energy estimate on (\ref{errfull}) and
a Gronwall argument shows that
$$
	\sup_{t\in [0,\tau/\eps]}\vert V(t)\vert_W
	\leq
 	\eps C(\tau,\vert U^0\vert_W)(1+\vert b\vert_W+\vert\nabla U^0\vert_W),
$$
where we also used the estimate
$
	\vert \pi_1(\uk+\eps D)U^0-U^0\vert_W\leq 
	\eps\cst(1+\vert b\vert_W+\vert\nabla U^0\vert_W)
$
(which is proved with the same arguments as (\ref{eq0est})). 
\end{proof}

\subsubsection{The nonlinear Schr\"odinger equation}\label{sectAs32}

In the Schr\"odinger approximation, the exact solution
to (\ref{eq1}) is approximated by $u^\eps_{app,2}$ defined as
\begin{equation}
	\label{uapp2}
	u_{app,2}^\eps(t,x)= 
	U_{(2)}(t,x)e^{i\frac{\uk\cdot x-\uom t}{\eps}}+\cc,
\end{equation}
where $U_{(2)}$ solves the \emph{nonlinear Schr\"odinger} equation
\begin{equation}
	\label{eqschrod}
	\left\lbrace
	\begin{array}{l}
	\dsp \dt U_{(2)}+\big(\nabla\omega_1(\uk)\cdot\nabla\big) U_{(2)}
	-\eps\frac{i}{2}\big(\nabla\cdot H_{\uk}(\omega_1)\nabla\big) U_{(2)}
	=\eps \pi_1(\uk)\cT(U_{(2)})\\
	U_{(2)}\,\init(x)=U^0(x),
	\end{array}\right.
\end{equation}
and where $H_{\uk}(\omega_1)$ stands for the Hessian of $\omega_1(\cdot)$ at $\uk$.
One then has
\begin{coro}[Schr\"odinger approximation]\label{coroschrod}
	Under the assumptions of Theorem \ref{th1}, 
	and for all $0<\eps<\eps_0$ ($\eps_0>0$ small enough), there exists
	a unique solution $U_{(2)}\in C([0,\tau_0/\eps);W(\R^d;\C^n))$
	to (\ref{eqschrod}).\\
	If moreover $U^0\in W^3(\R^d;\C^n)$ then 
	for all $0<\tau<\tau_0$, one also has
	$$
	\vert u_{ex}^\eps -u^\eps_{app,2}
	\vert_{L^\infty([0,\tau/\eps]\times \R^d)}
	\leq \eps C(\tau,\vert U^0\vert_W)(1+\vert \nabla U^0\vert_{W}+
	\vert b\vert_W
	+\vert {\mathfrak c}_{Schrod}\vert_\infty\vert U^0\vert_{W^3}),
	$$
	where $u^\eps_{app,2}$ is as defined in (\ref{uapp2}) 
	and 
	$$
	{\mathfrak c}_{Schrod}(\xi):=
	\frac{\big(\omega_1(\uk+\eps \xi)-\big(\uom
	+\eps\nabla\omega_1(\uk)\cdot\xi
	+\eps^2\frac{1}{2}\xi\cdot H_{\uk}(\omega_1)\xi\big)\big)}{\eps^3(1+\vert \xi\vert^3)}.		$$
\end{coro}
\begin{rema}\label{remschrod}
{\bf i.} A third order Taylor expansion of $\omega_1(\uk+\eps\xi)$ at $\xi=0$
shows that $\vert{\mathfrak c}_{Schrod}\vert_\infty$ 
is finite and can be bounded from 
above independently from $\varepsilon$.\\
{\bf ii.} The component $\vert {\mathfrak c}_{Schrod}\vert_\infty\vert U^0\vert_{W^3}$
of the error estimate does not appear for the full dispersion model. It is
due to the approximation of the nonlocal operator
$\frac{i}{\eps}(\omega_1(\uk+\eps D)-\uom)$ (lhs of (\ref{eqfull}))
by the differential
operator $\nabla\omega_1(\uk)\cdot\nabla U_{(2)}
	-\eps\frac{i}{2}\nabla\cdot H_{\uk}\nabla U_{(2)}$
(lhs of (\ref{eqschrod})). This error term is thus a \emph{linear} effect.\\
{\bf iii.} This additional term is responsible for
the bad behavior of the Schr\"odinger equation to model short pulses
(and more generally large spectrum waves such as chirped pulses). For
instance, for initial data like (\ref{IC}), the precision of the
Schr\"odinger approximation is of order 
$O(\eps(\frac{1}{\beta}+\frac{\vert {\mathfrak c}_{Schrod}\vert_\infty}{\beta^3}))$ when $\beta\ll1$. In order for the Schr\"odinger approximation to 
keep the same order of precision as the full dispersion model, one needs
therefore to have $\frac{\vert {\mathfrak c}_{Schrod}\vert_\infty}{\beta^3}\
\lesssim \frac{1}{\beta}$, which requires that  
$\beta^2\geq \vert {\mathfrak c}_{Schrod}\vert_\infty$. 
This condition is far much restrictive
than the practical rule $\beta\gg \eps$.\\
{\bf iv.} The Schr\"odinger approximation has been rigorously justified
for systems of the form (\ref{eq1}) (and also for some quasilinear
generalizations) by Donnat, Joly, M\'etivier and Rauch \cite{DJMR,JMR}
and Lannes \cite{Lannes} but in the Sobolev framework which we saw is not
adapted for the study of large spectrum pulses.
\end{rema}
\begin{proof}
As in the proof of Corollary \ref{corofull}, we focus on the error estimate and
omit the existence/uniqueness part of the proof.\\
The difference $V=U_{(1)}-U_{(2)}$ of the solution of the full dispersion
and Schr\"odinger equations 
solves the initial value problem
\begin{equation}
	\label{errschrod}
	\left\lbrace
	\begin{array}{l}
	\dt V+\frac{i}{\eps}(\omega_1(\uk+\eps D)-\uom)V
	=\eps \pi_1(\uk)\big(\cT(U_{(1)})-\cT(U_{(2)})\big)
	-\eps^2{\mathcal R}_2(D)U_{(2)}\\
	V\init(x)=0.
	\end{array}\right.
\end{equation}
where for all $\xi\in \R^d$,
$$
	{\mathcal R}_2(\xi)
	=\frac{1}{\eps^3}\big(i\omega_1(\uk+\eps \xi)-i\uom
	-i\eps\nabla\omega_1(\uk)\cdot\xi
	-\eps^2\frac{i}{2}\xi\cdot H_{\uk}(\omega_1)\xi\big).
$$
Remark now that one has for all $0\leq t\leq \tau/\eps$,
$$
	\vert {\mathcal R}_2(D)U_{(2)}(t)\vert_W\leq \vert 
	{\mathfrak c}_{Schrod}\vert_\infty
	\vert U_{2}(t)\vert_{W^3},
$$
with ${\mathfrak c}_{Schrod}(\cdot)$ as in the statement of the corollary;
differentiating the Schr\"odinger equation
(\ref{eqschrod}) and estimating the $W$-norm of the solution,
one also gets easily 
$$
	\sup_{t\in [0,\tau/\eps]}\vert U_{(2)}(t)\vert_{W^3}\leq C(\tau,\vert U^0\vert_W)(1+\vert  U^0\vert_{W^3}).
$$
Since the first term of the r.h.s. of (\ref{errschrod}) can be bounded
as in (\ref{eqsing1}), one gets from Gronwall's lemma applied to
(\ref{errschrod}) that
$$
	\sup_{t\in [0,\tau/\eps]}\vert V(t)\vert_W\leq
	C(\tau,\vert U^0\vert_W)(1+\vert {\mathfrak c}_{Schrod}\vert_\infty
	\vert U^0\vert_{W^3})
$$
which, together with Corollary \ref{corofull}, yields the result.
\end{proof}

\subsubsection{The nonlinear Schr\"odinger equation with improved dispersion
relation}\label{sectAs33}

As said in the introduction, we propose in this paper new approximations 
based on a family
of modified Schr\"odinger equations, whose dispersive
properties are closer to the exact model. Such an
approximation $u^\eps_{app,3}$ is defined as
\begin{equation}
	\label{uapp3}
	u_{app,3}^\eps(t,x)= 
	U_{(3)}(t,x)e^{i\frac{\uk\cdot x-\uom t}{\eps}}+\cc,
\end{equation}
where $U_{(3)}$ solves the \emph{nonlinear Schr\"odinger equation
with improved dispersion relation}: 
\begin{equation}
	\label{eqpade}
	\left\lbrace
	\begin{array}{l}
	\dsp \big(1-i\eps{\bf b}\cdot\nabla-\eps^2\nabla\cdot B\nabla\big)
	\dt U_{(3)}\\
	\dsp\indent
	+\big(\nabla\omega_1(\uk)\cdot\nabla-i\eps \nabla\cdot (\frac{1}{2}H_{\uk}(\omega_1)+\nabla\omega_1(\uk){\bf b}^T)\nabla
	+\eps^2 {\bf C}(\nabla)\big)U_{(3)}\\
	\dsp \indent=\eps\pi_1(\uk)\cT(U_{(3)})\\
	U_{(3)}\,\init(x)=U^0(x),
	\end{array}\right.
\end{equation}
where ${\bf b}\in \C^d$, $B\in {\mathcal M}_{d\times d}(\R)$ and 
${\bf C}:\C^d\times\C^d\times\C^d\to \C$
is a trilinear mapping. We assume moreover that
\begin{equation}\label{hyp}
B\mbox{ is symmetric positive}, \quad {\bf b}\in Range(B),
	\quad\mbox{ and }\quad
	4-{\bf b}\cdot (B^{-1}{\bf b})>0
\end{equation}
(note that even though $B^{-1}{\bf b}$ is not unique when $B$ is not definite,
the scalar ${\bf b}\cdot (B^{-1}{\bf b})$ is uniquely defined).
One then has the following result:
\begin{coro}[Improved Schr\"odinger approximation]\label{coropade}
	Under the assumptions of Theorem \ref{th1}, 
	and for all $0<\eps<\eps_0$ ($\eps_0>0$ small enough), 
	there exists
	a unique solution $U_{(3)}\in C([0,\tau_0/\eps);W(\R^d;\C^n))$
	to (\ref{eqpade}).\\
	If moreover $U^0\in W^3(\R^d;\C^n)$ then 
	for all $0<\tau<\tau_0$, one also has
	$$
	\vert u_{ex}^\eps -u^\eps_{app,3}
	\vert_{L^\infty([0,\tau/\eps]\times \R^d)}
	\leq \eps C(\tau,\vert U^0\vert_W)(1+\vert \nabla U^0\vert_{W}+
	\vert b\vert_W
	+\vert {\mathfrak c}_{improved}\vert_\infty\vert U^0\vert_{W^3}),
	$$
	where $u^\eps_{app,3}$ is as defined in (\ref{uapp3}) 
	and 
	$$
	{\mathfrak c}_{improved}(\xi):=
	\frac{\omega_1(\uk+\eps \xi)-\big(\uom
	+\eps\frac{\nabla\omega_1(\uk)\cdot\xi+\eps\xi\cdot(\frac{1}{2}H_{\uk}(\omega_1)+\nabla\omega_1(\uk){\bf b}^T)\xi-\eps^2{\bf C}(\xi)}{1+\eps{\bf b}\cdot \xi+\eps^2\xi\cdot B\xi}\big)
}{\eps^3(1+\vert \xi\vert^3)}.		$$
\end{coro}
\begin{rema}
{\bf i.} As for the Schr\"odinger equation, one can check by a simple Taylor
expansion that $\vert {\mathfrak c}_{improved}\vert_\infty$ 
is finite and uniformly
bounded with respect to $\eps$.\\
{\bf ii.} Taking ${\bf b}=0$, $B=0$ and ${\bf C}=0$ (this choice
satisfies (\ref{hyp})), one recovers the usual Schr\"odinger equation
(\ref{eqschrod}).\\
{\bf iii.} The interest of (\ref{eqpade}) with respect to (\ref{eqschrod})
is that one can choose ${\bf b}$, $B$ and ${\bf C}$ such that
${\mathfrak c}_{improved}\ll {\mathfrak c}_{Schrod}$, thus improving
considerably the accuracy of the approximation. In the one dimensional
case $d=1$, it is possible to choose ${\bf b}$, ${\bf B}$ and ${\bf C}$
in such a way that the dispersion relation
for (\ref{eqpade}) is the [3,2]-Pad\'e expansion of the dispersion
relation of (\ref{eqfull}). For the case of the Klein-Gordon system
(\ref{KG}), this leads to 
\begin{equation}\label{coeffpade}
	{\bf b}=\frac{2\uk}{{\bf v}^2+\uk^2},\qquad 
	{\bf B}=\frac{{\bf v}^2+4\uk^2}{4({\bf v}^2+\uk^2)^2}, \quad 
	{\bf C}=\frac{\uk(3{\bf v}^2+4\uk^2)}{4({\bf v}^2+\uk^2)^{5/2}};
\end{equation}
\begin{figure}\label{figrembsA}
	\includegraphics[width=6cm]{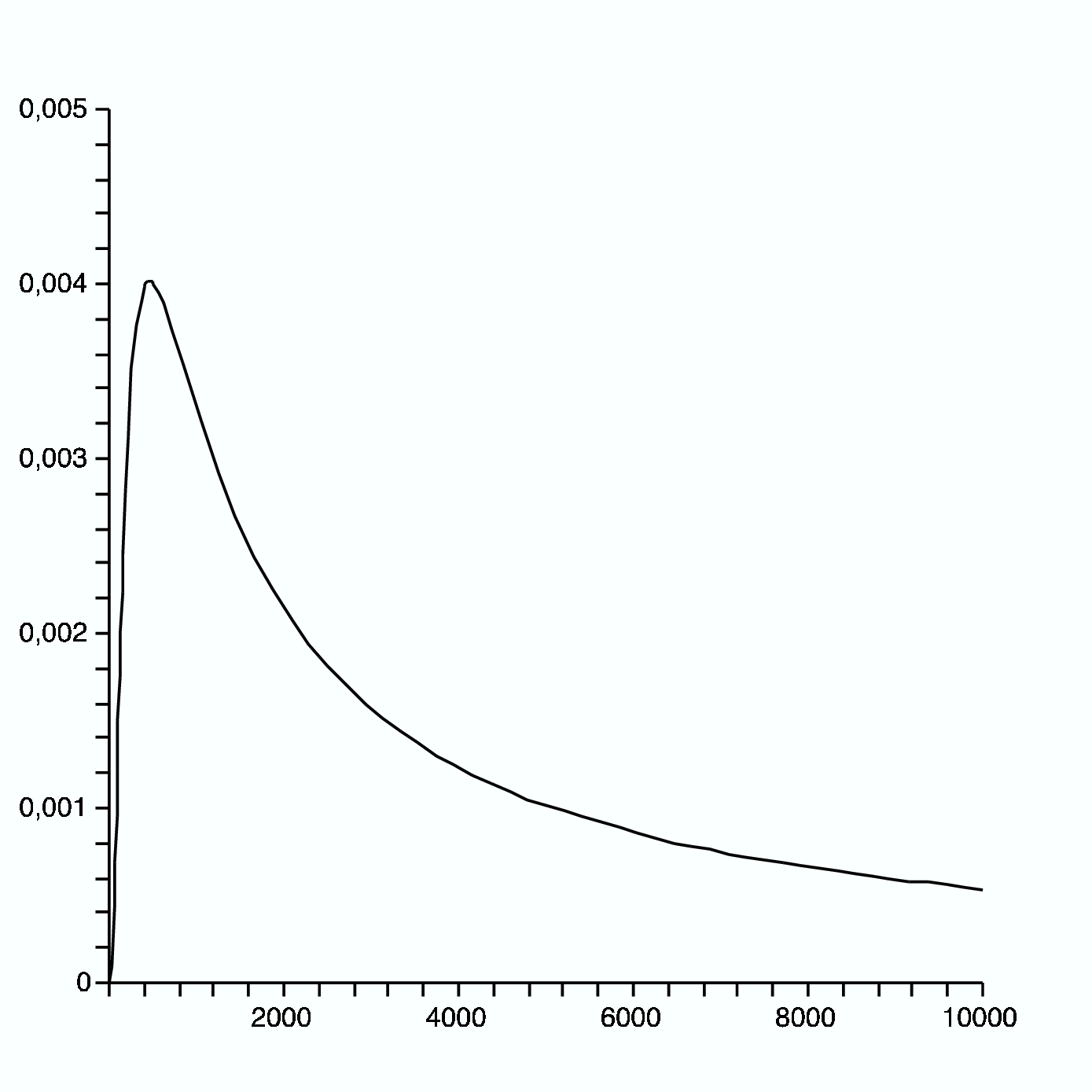}
	\caption{$\frac{{\mathfrak c}_{improved}(\xi)}
	{{\mathfrak c}_{Schrod}(\xi)}$ for $\eps=0.01$ with the coefficients
	(\ref{coeffpade}).}\label{figrembsA}

\end{figure}
we illustrate in Figure \ref{figrembsA} 
how much one gains by working with (\ref{eqpade}) instead of (\ref{eqschrod})
for the Klein-Gordon system (\ref{KG}) with ${\bf v}=\uk=1$.\\

{\bf iv.} The same analysis as is Remark \ref{remschrod}.iii shows that
the approximation provided by (\ref{eqpade}) is of the same
order as the envelope approximation if $\beta^2\geq \vert {\mathfrak c}_{improved}\vert_\infty$.
Since $\vert{\mathfrak c}_{improved}\vert_\infty\ll \vert{\mathfrak c}_{Schrod}\vert_\infty$, this condition is much weaker than the corresponding one
for the usual Schr\"odinger model. In some particular cases, this
condition can even be weaker than the ``practical rule'' $\eps\ll \beta$.
\end{rema}
\begin{proof}
Choosing  $\xi_0\in-\frac{1}{2}B^{-1}{\bf b}$, one can check that 
$$
	1+{\bf b}\cdot\xi+\xi\cdot B\xi=
	1-\frac{1}{4}{\bf b}\cdot(B^{-1}{\bf b})+(\xi-\xi_0)\cdot B(\xi-\xi_0),
$$
so that it follows from assumption (\ref{hyp}) that
$1+{\bf b}\cdot\xi+\xi\cdot B\xi>0$ (uniformly with respect to
$\xi\in\R^d$). The operator $1-\eps i{\bf b}-\eps^2\nabla\cdot B\nabla$
is therefore invertible, and its inverse is the Fourier
multiplier $(1+\eps {\bf b}\cdot D+\eps^2 D\cdot BD)^{-1}$. The equation
(\ref{eqpade}) can therefore be rewritten as
$$
	\begin{array}{l}
	\dsp 
	\dt U_{(3)}
	+i\frac{\nabla\omega_1(\uk)\cdot D+\eps D\cdot (\frac{1}{2}H_{\uk}(\omega_1)+\nabla\omega_1(\uk){\bf b}^T)D
	-\eps^2 {\bf C}(D)}{(1+\eps {\bf b}\cdot D+\eps^2 D\cdot BD)}U_{(3)}
	\vspace{1mm}\\
	\dsp \indent\indent=(1+\eps {\bf b}\cdot D+\eps^2 D\cdot BD)^{-1}\pi_1(\uk)\cT(U_{(3)}).
	\end{array}
$$
Since  $(1+\eps {\bf b}\cdot D+\eps^2 D\cdot BD)^{-1}$ is
regularizing (of order $-2$) and acts on $W(\R^d;\C^n)$ uniformly
with respect to $\eps>0$, the proof of the result follows 
exactly the same lines as the proof of Corollary \ref{coropade} and
we thus omit it.
\end{proof}

\section{Numerical simulations}\label{sectNum}

\subsection{The equations}

This section is devoted to the comparison of solutions of the different 
asymptotic equations derived in Section \ref{sectAs} with the solutions
of the full system (\ref{eq1}). In particular, we want to check
numerically the results proved in Corollaries \ref{corofull}, \ref{coroschrod}
and \ref{coropade}.\\
We consider both short pulses (\S \ref{sectSP}) and chirped pulses
(\S \ref{sectchirp}); for the numerical computations, we use the toy
model (\ref{KG}) with dimension $d=1$ and ${\bf v}=1$ (this is also
the model used in \cite{CGL} for the study of chirped pulses). 
Writing $u=(f,g)^T$, this model reads
\begin{equation}\label{toys}
\left\{ \begin{array}{l}
\dsp\partial_t f+\partial_x g -\frac{g}{\varepsilon}  =-\varepsilon(|f|^2+|g|^2)g,\\
 \dsp\partial_t g+\partial_x f +\frac{f}{\varepsilon}  =\varepsilon(|f|^2+|g|^2)f.
\end{array}\right.
\end{equation}
The initial conditions are taken of the form
\begin{equation}\label{ICtoy}
	f_{\vert_{t=0}}(x)=f^0(x)e^{i\frac{x}{\eps}}+\cc,\qquad
	g_{\vert_{t=0}}(x)=\frac{1-i}{\sqrt{2}}f^0(x)e^{i\frac{x}{\eps}}+\cc
\end{equation}
(which, with the notations of Theorem \ref{th1}, corresponds to $\uk=1$, $b=0$ and $a=\pi_1(\uk)a=(f^0,\frac{1-i}{\sqrt{2}}f^0)^T$).\\
The asymptotic models derived in Section \ref{sectAs} read in this
particular case as follows:
\begin{itemize}
\item \emph{The full dispersion model.} The exact solution to (\ref{toys})
is approximated  by $$u^\varepsilon_{app,1}(t,x)=U_{(1)}(t,x)e^{i\frac{x -\sqrt{2} t}{\varepsilon}}+\cc,$$
where $U_{(1)}=(f_{(1)},g_{(1)})$ solves (recall that $D=-i\partial_x$)
\begin{align}\label{toyfdm}
\dsp \left\{ \begin{array}{l}\partial_t  f_{(1)} +\frac{i}{\varepsilon}\frac{- D+\varepsilon D^2}{\sqrt{1+(1+\varepsilon D)^2}+\sqrt{2}} f_{(1)} =\frac{4i\varepsilon}{\sqrt{2}} |f_{(1)}|^2f_{(1)},\vspace{0.5mm}\\
g_{(1)}= \frac{1-i}{\sqrt{2}}f_{(1)}\;\text{ (polarization condition)}.\end{array}\right.
\end{align}
\item \emph{The nonlinear Schr\"odinger equation.} The exact solution to
(\ref{toys}) is approximated by 
$$u^\varepsilon_{app,2}(t,x)=U_{(2)}(t,x)e^{i\frac{x -\sqrt{2} t}{\varepsilon}}+\cc,$$
where $U_{(2)}=(f_{(2)},g_{(2)})$ solves
\begin{align}\label{toyschro}
\dsp \left\{ \begin{array}{l}\partial_t  f_{(2)} +\frac{1}{\sqrt{2}}\big(\partial_x f_{(2)} -\frac{\eps}{4}\partial_x^2 f_{(2)}\big) =\eps\frac{4i}{\sqrt{2}} |f_{(2)}|^2f_{(2)},\vspace{0.5mm}\\
g_{(2)}= \frac{1-i}{\sqrt{2}}f_{(2)}\;\text{ (polarization condition)}.\end{array}\right.
\end{align}
\item \emph{The nonlinear Schr\"odinger equation with improved dispersion relation.}
We approximate the solution of $\eqref{toys}$ by
$$u^\varepsilon_{app,3}(t,x)=U_{(3)}(t,x)e^{i\frac{x -\sqrt{2}t}{\varepsilon}}+\cc,$$
where $U_{(3)}=(f_{(3)},g_{(3)})$ solves
\begin{align}\label{toypade}
\dsp \left\{ \begin{array}{l}\big(1-i\varepsilon \partial_x -\varepsilon^2 \frac{5}{16}\partial_x^2\big) \partial_t  f_{(3)} +\frac{1}{\sqrt{2}}\big(\partial_x-i\varepsilon \frac{5}{4} \partial_x^2 +\varepsilon^2 \frac{7}{16}\partial_x^3\big)f_{(3)}=\frac{4i\varepsilon}{\sqrt{2}} |f_{(3)}|^2f_{(3)},\vspace{0.5mm}\\
g_{(3)}= \frac{1-i}{\sqrt{2}}f_{(3)}\;\text{ (polarization condition)};\end{array}\right.
\end{align}
note that this modified Schr\"odinger equation corresponds 
to the set of coefficients (\ref{coeffpade}).
\end{itemize}

\subsection{The numerical scheme}
We use a spectral method in space and a splitting technique in time for
all the equations introduced in the previous section. We give here some
details on the numerical scheme used for (\ref{toys});
for  (\ref{toyfdm}), (\ref{toyschro})
and (\ref{toypade}), we use straightforward adaptations of this scheme.\\
Let  us denote by $S_L(t)$ and $S_{NL}(t)$ the evolution operator associated
respectively to the linear and nonlinear part of (\ref{toys}); namely,
$$
	S_L(t)u^0=(f(t),g(t)), 
	\mbox{ with }
	\left\{ \begin{array}{l}
\dsp\partial_t f+\partial_x g -\frac{g}{\varepsilon}  =0\\
 \dsp\partial_t g+\partial_x f +\frac{f}{\varepsilon}  =0
\end{array}\right.
\mbox{ and } 
	(f(0),g(0))=u^0,
$$
and
$$
	S_{NL}(t)u^0=(f(t),g(t)), 
	\mbox{ with }
	\left\{ \begin{array}{l}
\dsp\partial_t f=-\varepsilon(|f|^2+|g|^2)g\\
 \dsp\partial_t g=\varepsilon(|f|^2+|g|^2)g
\end{array}\right.
\mbox{ and } 
	(f(0),g(0))=u^0
$$
(and with periodic boundary conditions).\\
The numerical computation of $S_L(t)$ is made through an FFT-based
spectral method while an explicit integration is used for $S_{NL}(t)$; we
then use a second order splitting scheme to compute $u^{n+1}\sim u((n+1)\Delta t)$ in terms of $u^n\sim u(n\Delta t)$ 
(where $\Delta t$ denotes the time step):
$$
	u^{n+1}=S_L(\frac{\Delta t}{2})
	S_{NL}(\Delta t)S_{L}(\frac{\Delta t}{2})u^n.
$$

\subsection{Numerical results for short pulses}\label{sectSP}
In this section, we are interested in short pulses, 
that is we consider initial conditions for
(\ref{toys}) of  the form (\ref{ICtoy}), with
$$f^0(x)=G(\frac{x-x_0}{\beta}),$$
where $G$ is a smooth function. In the present numerical
computations, the computational domain is $[0,L]$ with $L=30\pi$, and
we take $x_0=15$ and $G(x)=e^{-x^2}$.

The accuracy of the approximations  (\ref{toyfdm}), (\ref{toyschro})
and (\ref{toypade}) is checked using the following quantity:
\begin{equation}\label{err}
	E_{(j)}(\eps,\beta)=\sup_{t\in [0,\frac{1}{\eps}]}\frac{|f(t,\cdot)-(f_{(j)}(t,\cdot)e^{i\frac{kx-\omega t}{\varepsilon}}+\cc)|_{\infty}}{|f(t,\cdot)|_{\infty}},
\end{equation}
where $j=1$ for the full dispersion model (\ref{toyfdm}), $j=2$
for the usual Schr\"odinger approximation (\ref{toyschro}), and
$j=3$ for our new modified Schr\"odinger equation (\ref{toypade}).\\
The exact solution and the difference between the exact solution with the approximation furnished by
the FD, Schr\"odinger and improved Schr\"odinger models
are plotted in Figure \ref{figsolSP} for $\eps=0.01$, $\beta=0.075$ and at time
$T=50$ on the domain $x\in [0,30\pi]$.
\begin{figure}
	\includegraphics[width=4cm]{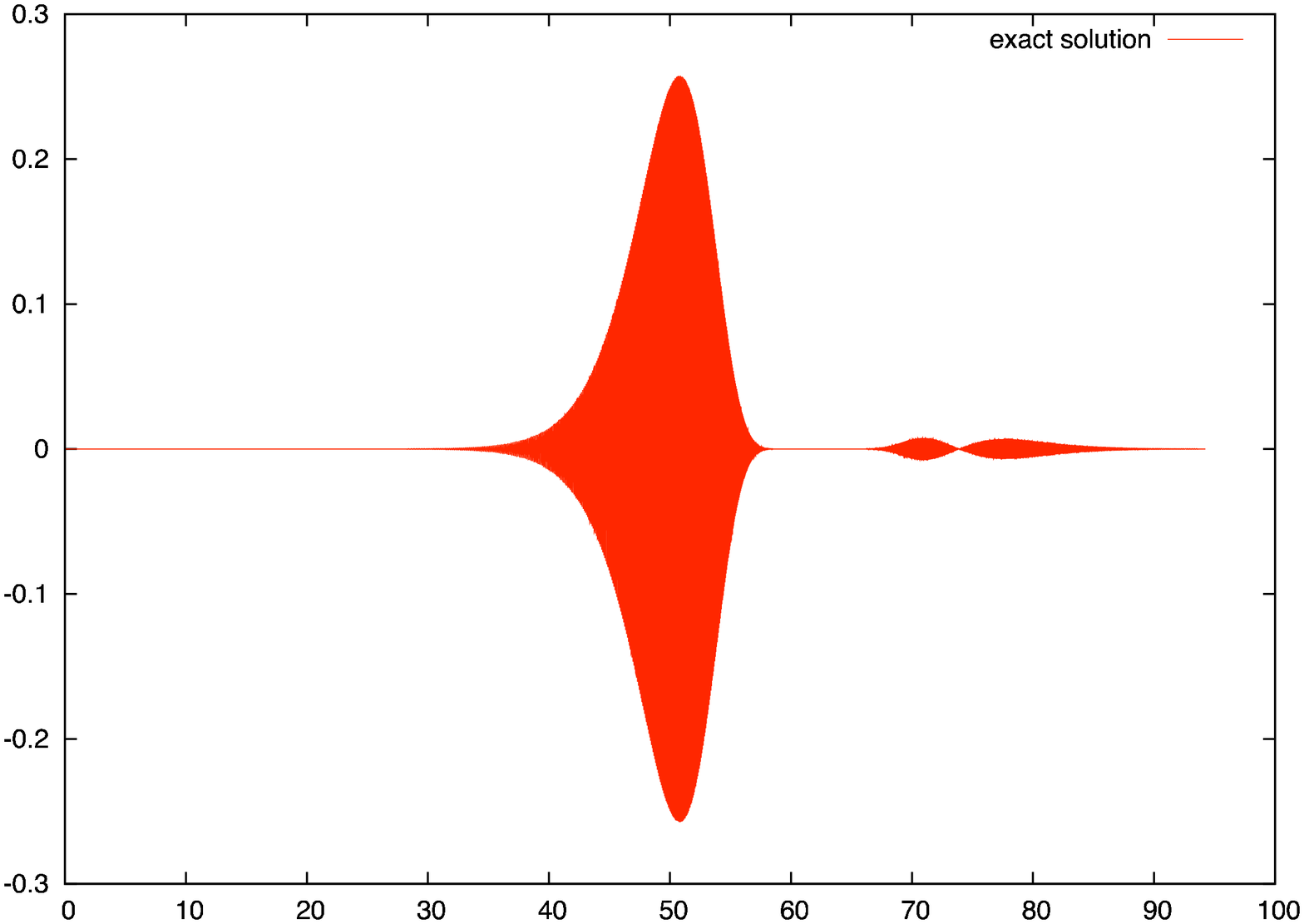}
	\includegraphics[width=4cm]{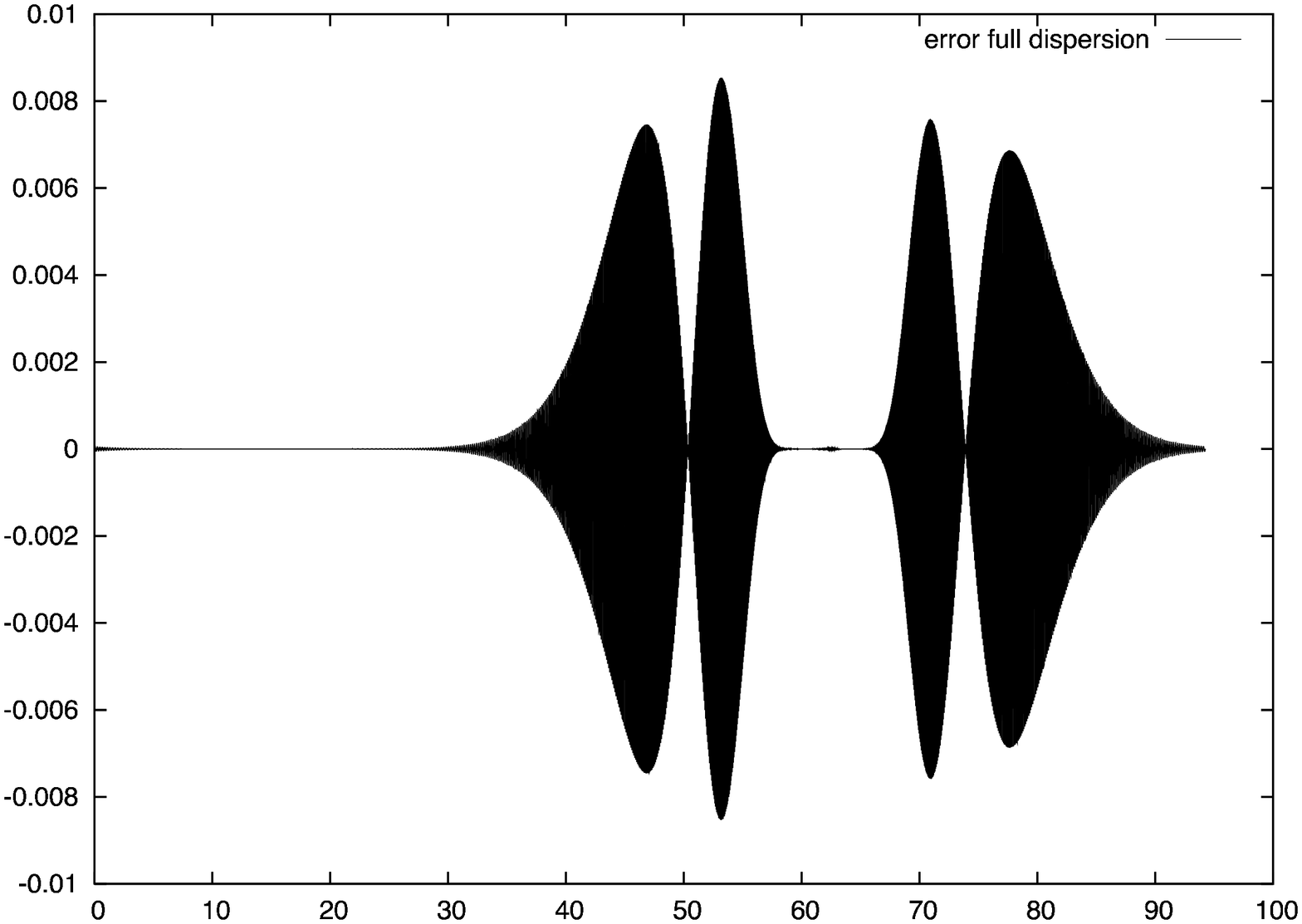}\\
	\includegraphics[width=4cm]{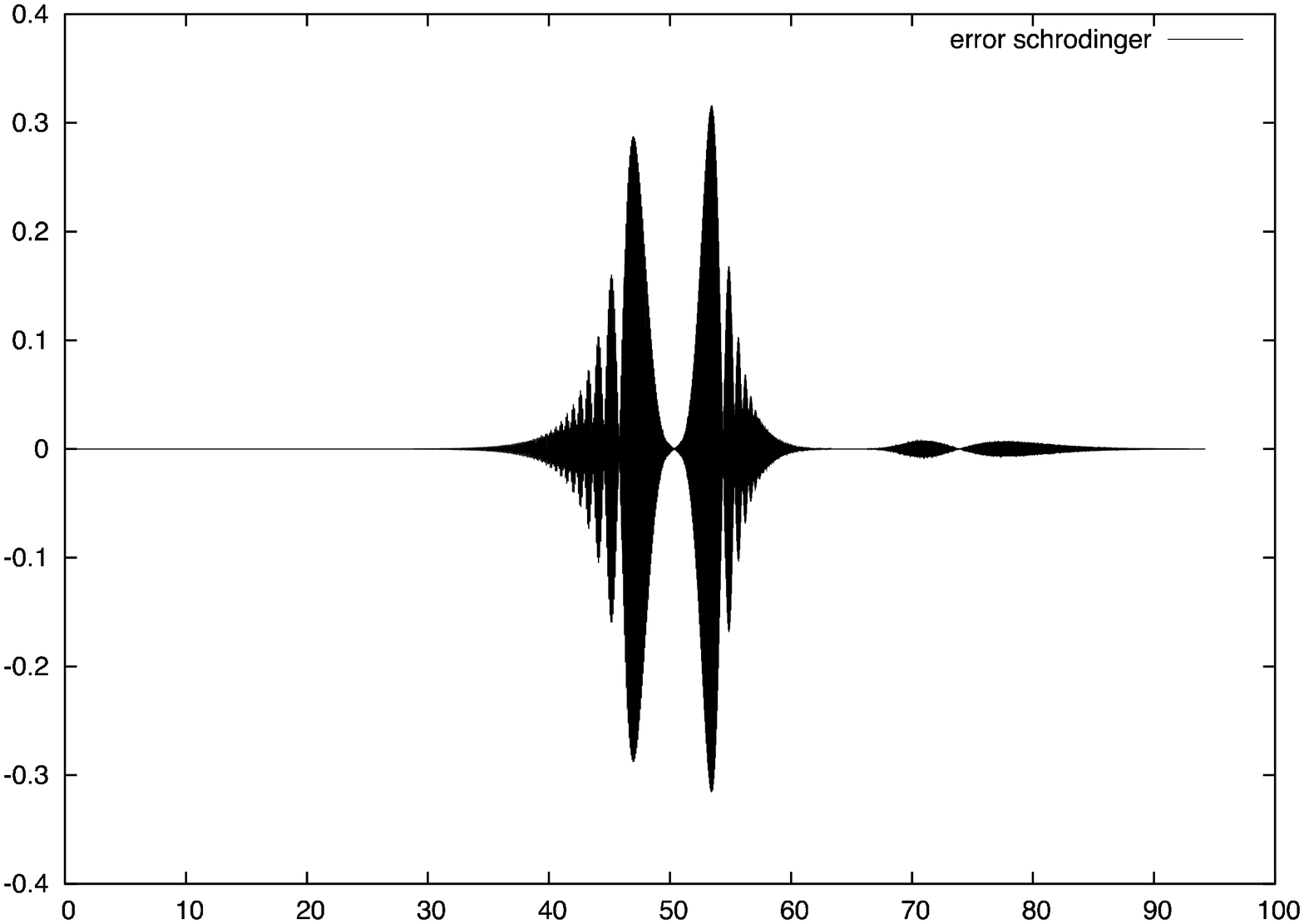}
	\includegraphics[width=4cm]{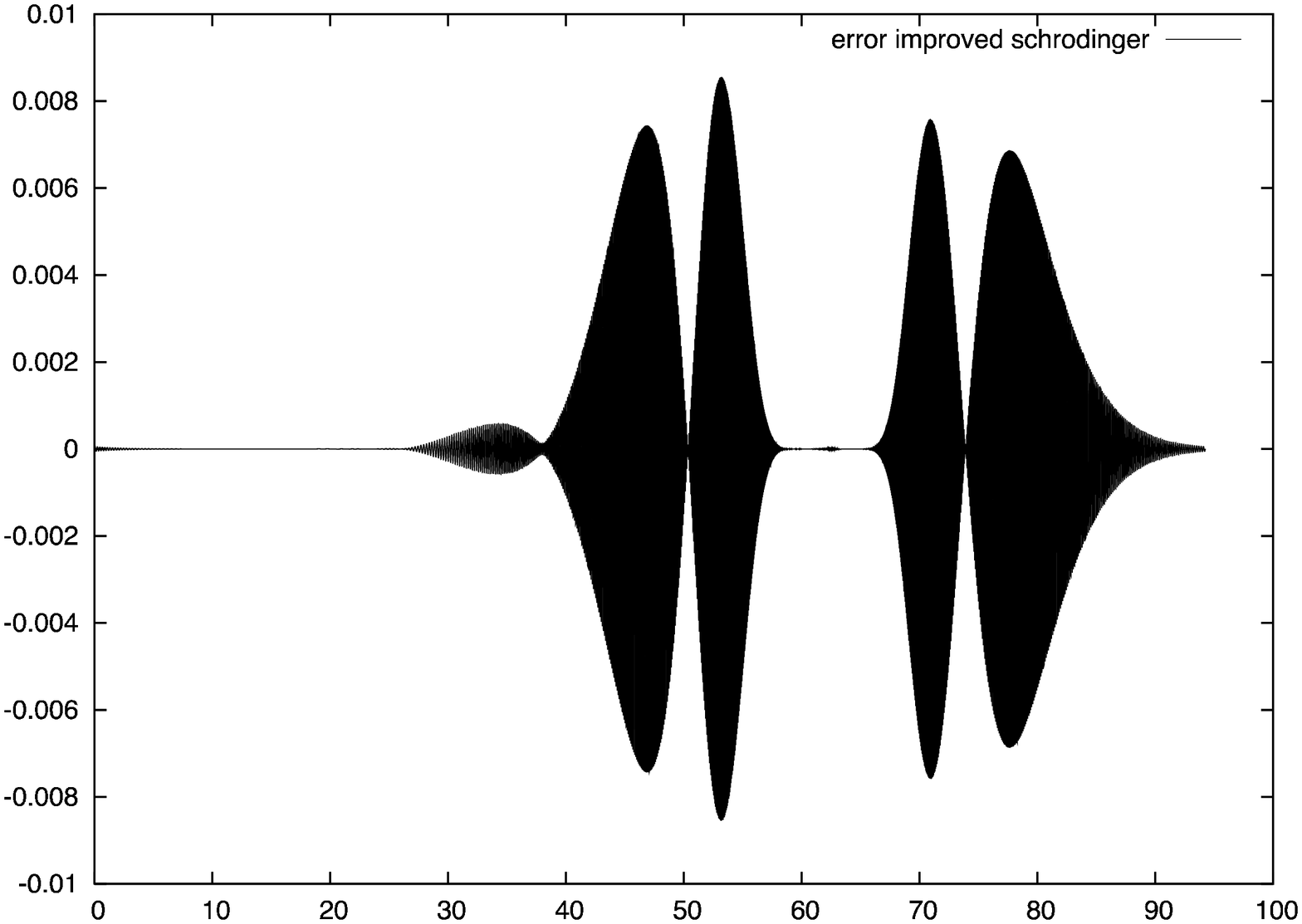}
	\caption{Short pulses: the exact solution, and the difference between the exact solution and 
 the FD, Schr\"odinger and improved Schr\"odinger models
	(from left to right and top to bottom) with $\eps=0.01$, $\beta=0.075$
	and $T=50$.} \label{figsolSP}
\end{figure}
The following computations are also performed
to test the accuracy of the approximate models:\\
$\bullet$ \emph{Test 1:} With $\beta=1$ fixed, we let $\eps$ vary from
$\eps=0.001$ to $\eps=0.1$. This configuration corresponds to usual
wave packets for which the three models should have a comparable
accuracy of $O(\eps)$ when $\eps$ is small enough. One can indeed
observe on Figure \ref{fig1} that the errors $E_{(j)}(\eps,\beta)$ ($j=1,2,3$)
grow linearly with $\eps$. One will also check that when $\eps$ is too large
($\eps\sim 5.10^{-2}$ for a rough precision of $20\%$), none
of the models furnishes a good approximation.
\begin{figure}
	\includegraphics[width=6cm]{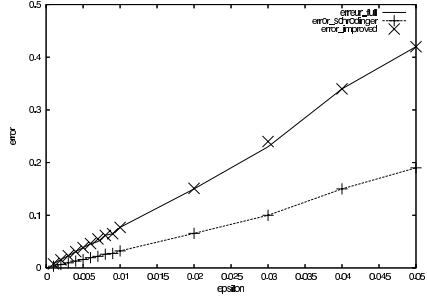}
	\caption{The errors $E_{(j)}(\eps,\beta)$ for $\beta=1$ and
	$\eps\in[0.001,0.1]$; $j=1$ corresponds to FD, $j=2$ to Schr\"odinger and $j=3$ to the improved Schr\"odinger.} \label{fig1}
\end{figure}

\noindent
$\bullet$ \emph{Test 2:} Here, we look at the same configuration as in
Test 1 but with $\beta=0.1$, that is, we investigate here short pulses.
We can observe on Figure \ref{fig2} that the FD and improved Schr\"odinger
models provide a good approximation, but that the usual Schr\"odinger
approximation is completely inaccurate.
\begin{figure}
	\includegraphics[width=6cm]{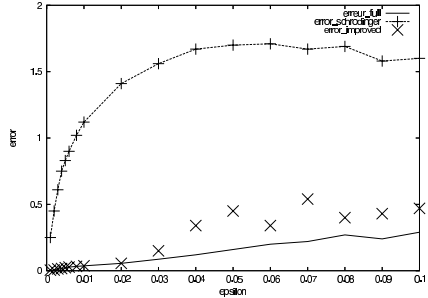}
	\caption{Short pulses: The errors $E_{(j)}(\eps,\beta)$ for $\beta=0.1$ and
	$\eps\in[0.001,0.1]$; $j=1$ corresponds to FD, $j=2$ to Schr\"odinger and $j=3$ to the improved Schr\"odinger.} \label{fig2}
\end{figure}

\noindent
$\bullet$ \emph{Test 3:} Here, $\eps=0.01$ is fixed and we let $\beta$ vary 
from $\beta=0.01$ (short pulses) to $\beta=1$ (wave packets). 
It can be checked that the FD model furnishes a correct approximation for
$\beta\gtrsim 0.03$ and that for such values of $\beta$, the improved Schr\"odinger approximation has the same precision. This is to be contrasted with
the usual Schr\"odinger approximation which is completely inaccurate until
$\beta\sim 0.2$.
\begin{figure}
	\includegraphics[width=6cm]{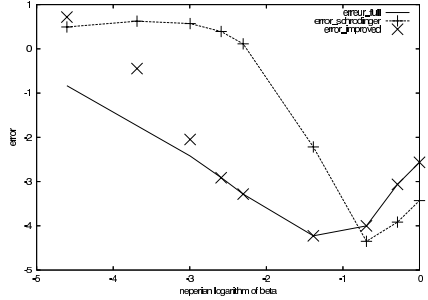}
	\caption{Short pulses: 
	The errors $E_{(j)}(\eps,\beta)$ for $\eps=0.01$ and
	$\beta\in[0.01,1]$; $j=1$ corresponds to FD, $j=2$ to Schr\"odinger and $j=3$ to the improved Schr\"odinger.} \label{fig3}
\end{figure}

\subsection{Numerical results for chirped pulses}\label{sectchirp}

In this section, we are interested in chirped pulses, 
that is we consider initial conditions for
(\ref{toys}) of  the form (\ref{ICtoy}), with
$$f^0(x)=G(x-x_0)\cos(\frac{1}{\beta}\cos(\frac{x-x_0}{\beta})),$$
where $G$ is a smooth function. In the present numerical
computations, the computational domain is $[0,L]$ with $L=30\pi$, and
we take $x_0=15$ and $G(x)=e^{-x^2}$.

The accuracy of the approximations  (\ref{toyfdm}), (\ref{toyschro})
and (\ref{toypade}) is checked using the quantities
$E_{(j)}(\eps,\beta)$ ($j=1,2,3$) defined in (\ref{err}).\\
The exact solution and the diffenrence between the exact solution
and the FD, Schr\"odinger and improved Schr\"odinger models
are plotted in Figure \ref{figsolCP} for $\eps=0.01$, $\beta=0.3$ and at time
$T=1/\eps=100$.
\begin{figure}
	\includegraphics[width=4cm]{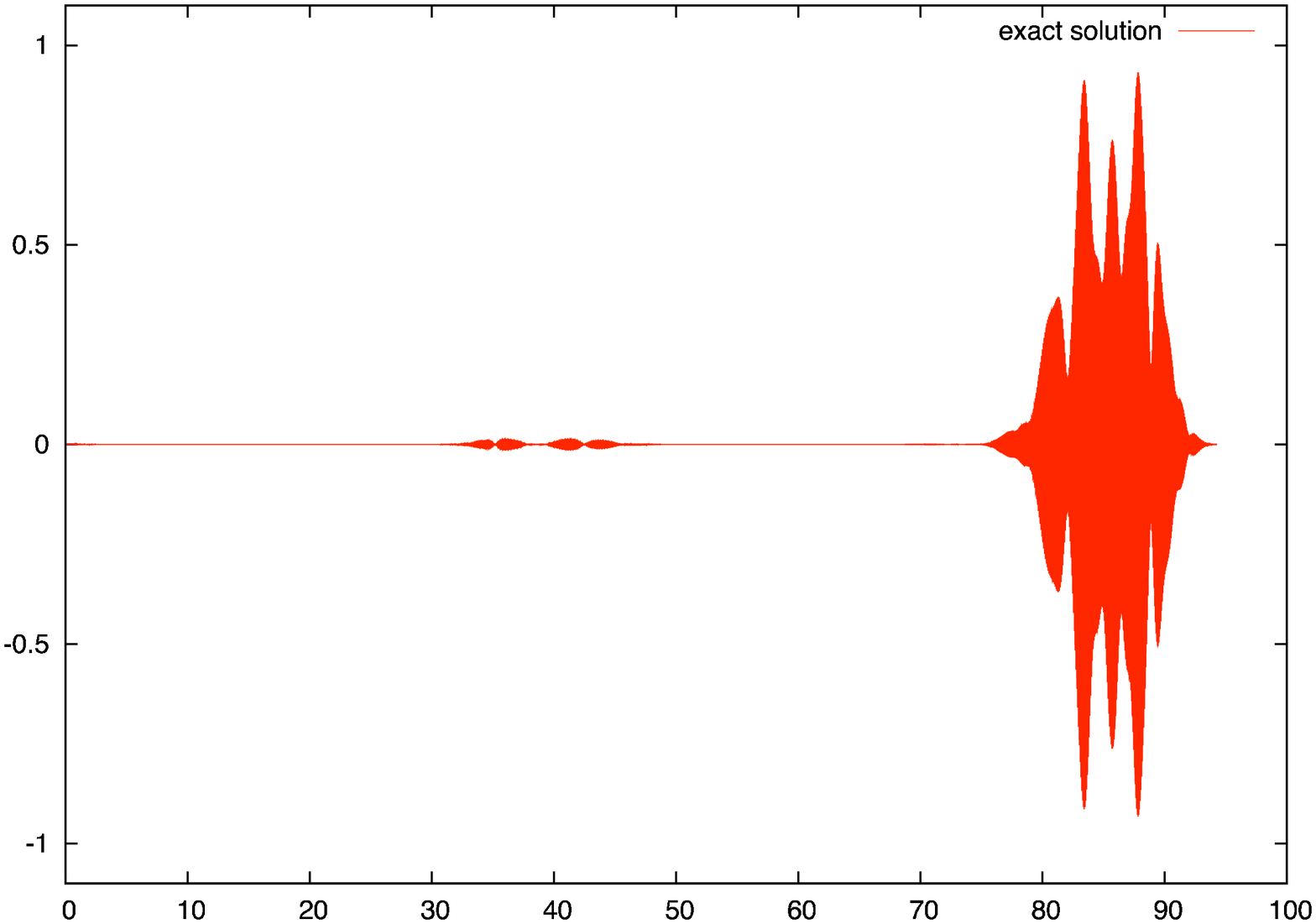}
	\includegraphics[width=4cm]{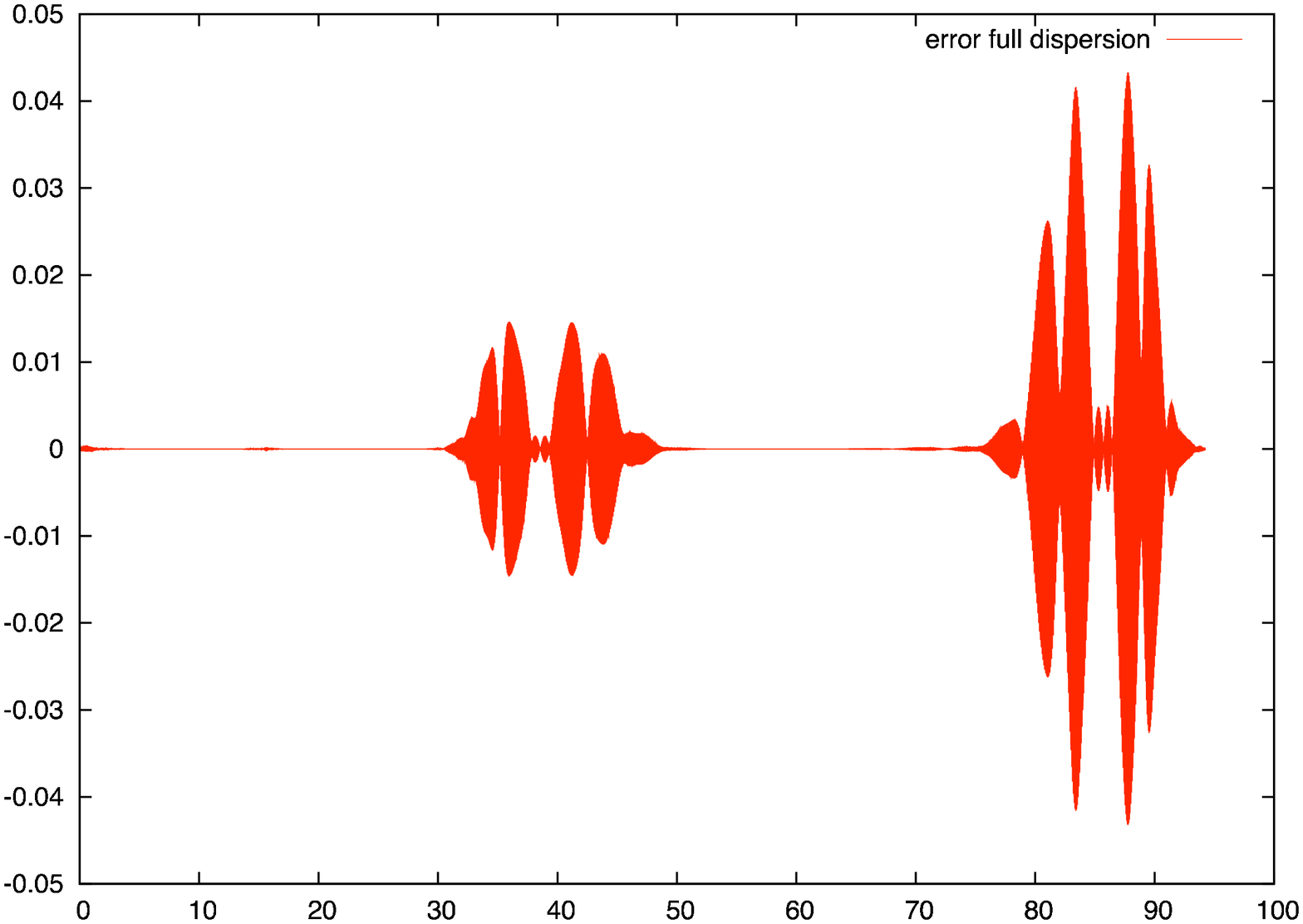}\\
	\includegraphics[width=4cm]{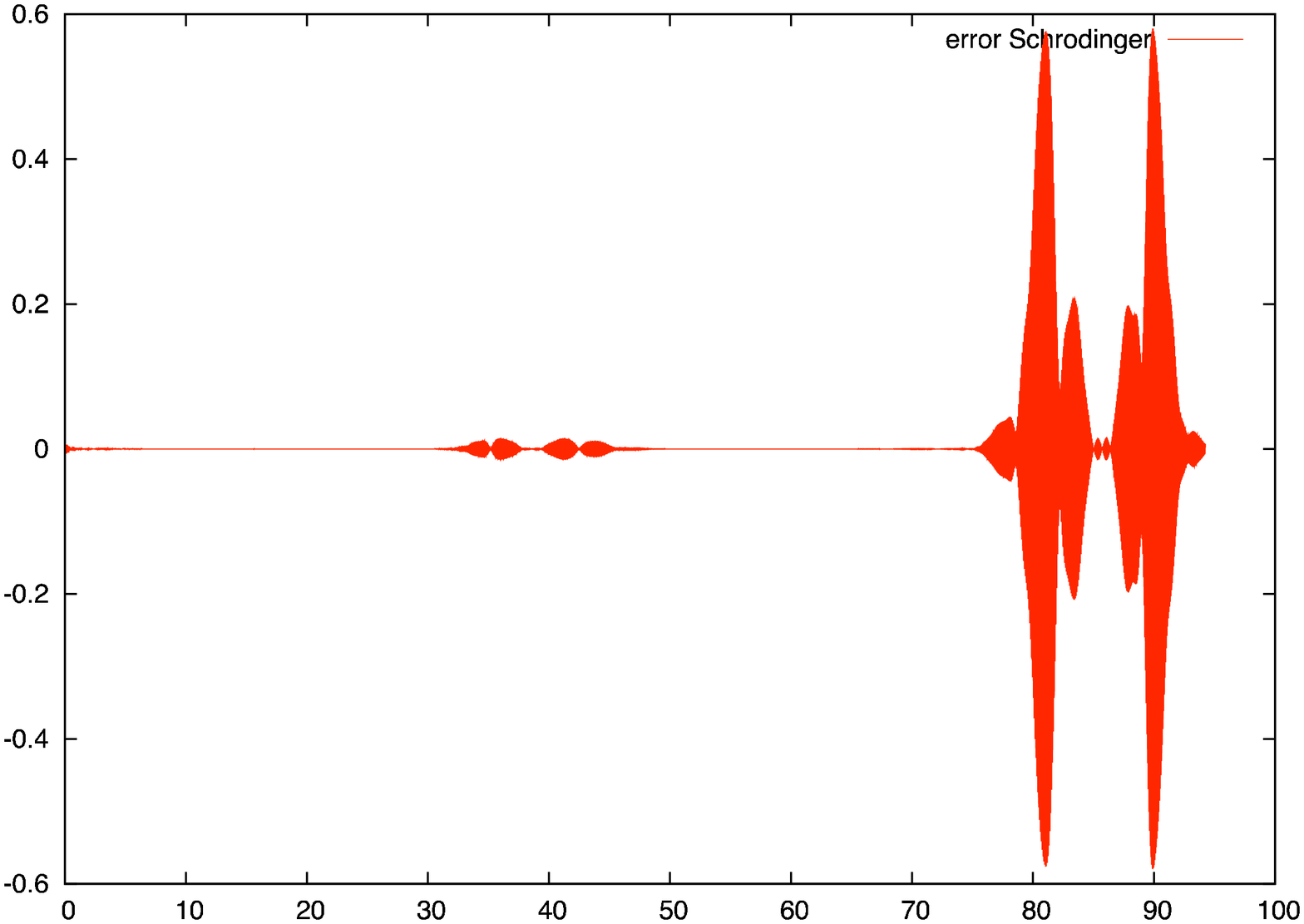}
	\includegraphics[width=4cm]{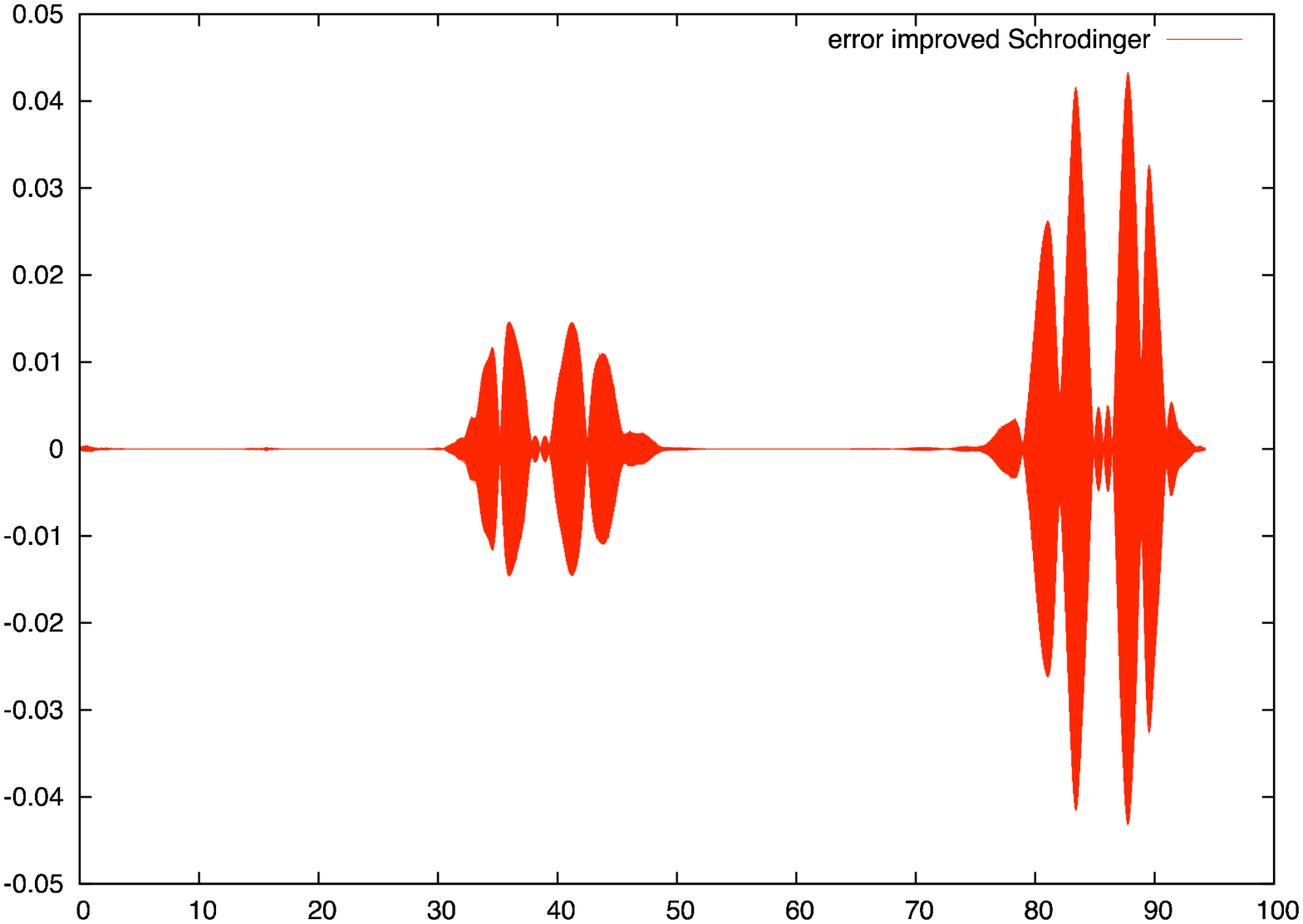}
	\caption{Chirped pulses: the exact solution, and the 
	difference between 
	the exact solution and the 
	FD, Schr\"odinger and improved Schr\"odinger models
	(from left to right and top to bottom) with $\eps=0.01$, $\beta=0.3$
	and $T=100$.} \label{figsolCP}
\end{figure}
The following computations are also performed
to test the accuracy of the approximate models:\\
$\bullet$ \emph{Test 1:} With $\beta=0.1$ fixed, we let $\eps$ vary from
$\eps=0.001$ to $\eps=0.1$. 
We can observe on Figure \ref{fig4} that the FD and the improved Schr\"odinger models are good approximations for $\varepsilon \leq 0.003$. Above this value, the approximation is no longer pertinent. Furthermore, the classical Schr\"odinger model is inapropriate for this range of parameters.
\begin{figure}
	\includegraphics[width=6cm]{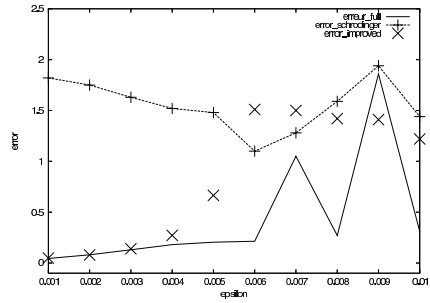}
	\caption{Chirped pulses: The errors $E_{(j)}(\eps,\beta)$ for $\beta=0.01$ and
	$\eps\in[0.001,0.1]$; $j=1$ corresponds to FD, $j=2$ to Schr\"odinger and $j=3$ to the improved Schr\"odinger.} \label{fig4}
\end{figure}

\noindent
$\bullet$ \emph{Test 2:} Here, $\eps=0.01$ is fixed and we let $\beta$ vary 
from $\beta=0.01$ (chirped pulses) to $\beta=1$ (wave packets). We 
observe on Figure \ref{fig5} that both FD and improved Schr\"odinger models become appropriate for $\beta\geq 0.1$ whereas the Schr\"odinger approximation is acceptable for $\beta \geq 0.4$. 
\begin{figure}
	\includegraphics[width=6cm]{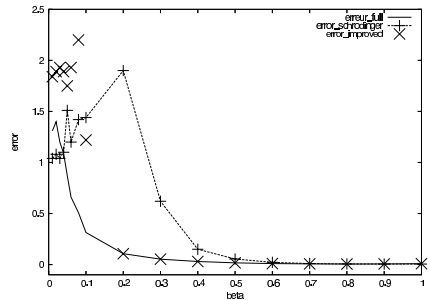}
	\caption{Chirped pulses: The errors $E_{(j)}(\eps,\beta)$ for $\eps=0.01$ and
	$\beta\in[0.01,1]$; $j=1$ corresponds to FD, $j=2$ to Schr\"odinger and $j=3$ to the improved Schr\"odinger.} \label{fig5}
\end{figure}

\end{document}